\documentclass[11pt]{article}

\usepackage{amsmath}
\usepackage{amsfonts}
\usepackage{amsthm}
\usepackage{amssymb}
\usepackage{MnSymbol}
\usepackage{bbm}
\usepackage{cite}
\usepackage{enumerate}
\usepackage{authblk}
\usepackage[top=1in,bottom=1in,left=1in,right=1in]{geometry}

\def\sgn{\mathrm{sgn}}

\newcommand{\1}{\mathbbm{1}}

\newcommand{\eps}{\epsilon}
\newcommand{\R}{\mathbb{R}}

\newcommand{\N}{\mathbb{N}}
\newcommand{\E}{\mathbb{E}}
\renewcommand{\P}{\mathbb{P}}
\newcommand{\calA}{\mathcal{A}}

\newcommand{\calF}{\mathcal{F}}
\newcommand{\calT}{\mathcal{T}}
\newcommand{\calG}{\mathcal{G}}
\newcommand{\calM}{\mathcal{M}}

\newcommand{\SAW}{\Gamma^{\mathrm{SAW}}}
\newcommand{\NB}{\Gamma^{\mathrm{NB}}}
\newcommand{\sZeta}{\mathcal{Z}^{\mathrm{short}}}
\newcommand{\lZeta}{\mathcal{Z}^{\mathrm{long}}}
\newcommand{\Vints}{{V_{\ge 3}}}
\newcommand{\Vends}{{V_{\mathrm{end}}}}

\DeclareMathOperator{\Var}{Var}

\DeclareMathOperator{\Binom}{Bin}

\DeclareMathOperator{\Poisson}{Poisson}
\newcommand{\badpaths}{\Gamma^\text{bad}}

\newtheorem{theorem}{Theorem}[section]
\newtheorem{lemma}[theorem]{Lemma}
\newtheorem{corollary}[theorem]{Corollary}
\newtheorem{proposition}[theorem]{Proposition}
\newtheorem{definition}[theorem]{Definition}
\newtheorem{conjecture}[theorem]{Conjecture}

\newtheorem{claim}[theorem]{Claim}
\newtheorem{assumption}[theorem]{Assumption}

\newcommand{\longIn}{m_\text{in}}
\newcommand{\longOut}{m_\text{out}}

\title{A Proof Of The Block Model Threshold Conjecture}
\author{Elchanan Mossel\thanks{
U.C. Berkeley. Supported by NSF grant DMS-1106999, NSF grant CCF 1320105  and DOD ONR grant N000141110140},
\ Joe Neeman\thanks{
U.T. Austin and the University of Bonn. Supported by NSF grant DMS-1106999 and DOD ONR grant N000141110140},
\ and
Allan Sly\thanks{
U.C. Berkeley and the Australian National University. Supported by an
Alfred Sloan Fellowship and NSF grant DMS-1208338.
}}

\begin{document}
\maketitle

\maketitle

\begin{abstract}
We study a random graph model called the ``stochastic block model'' in statistics and the
``planted partition model'' in theoretical computer science.
In its simplest form, this is a random graph
with two equal-sized classes of vertices, with a within-class edge probability of $q$
and a between-class edge probability of $q'$.

A striking conjecture of Decelle, Krzkala, Moore and Zdeborov\'a~\cite{Z:2011}, based on deep, non-rigorous ideas from statistical physics,
gave a precise prediction for the algorithmic threshold of clustering in the sparse planted partition model.
In particular, if $q = a/n$ and $q' = b/n$, $s=(a-b)/2$ and $d=(a+b)/2$ then Decelle et al.\ conjectured
that it is possible to efficiently cluster in a way correlated with the true partition if
$s^2 > d$ and impossible if
$s^2 < d$.
By comparison, until recently the best-known rigorous result  showed that clustering is possible if $s^2 > C d \ln d$ for sufficiently large~$C$.

In a previous work, we proved that indeed it is information theoretically impossible to cluster if $s^2 \le d$ and moreover that  it is information theoretically impossible to even estimate the model parameters from the graph when $s^2 < d$.
Here we prove the rest of the conjecture by providing an efficient algorithm for clustering in a way that is correlated with the true partition when $s^2 > d$. A different independent proof of the same result was recently obtained by Massouli\'e~\cite{Massoulie:13}.

\newpage
\end{abstract}

\section{Introduction}

\subsection{The stochastic block model}
We consider the simplest version of the stochastic block model, namely the version with two symmetric states:
\begin{definition}[The stochastic block model]
  For $n \in \N$ and $q, q' \in (0, 1)$, let $\calG(n, q, q')$
  denote the model of random, $\pm 1$-labelled graphs on $n$ vertices in which
  each vertex $u$ is assigned (independently and uniformly at random)
  a label $\sigma_u \in \{1, -1\}$, and then each possible edge
  $\{u, v\}$ is included with probability $q$ if $\sigma_u = \sigma_v$
  and with probability $q'$ if $\sigma_u \ne \sigma_v$.
\end{definition}
If $q = q'$, the stochastic block model is just an
Erd\H{o}s-R\'enyi model, but if $q \gg q'$ then one expects that a typical
graph will have two well-defined clusters.

Theoretical computer scientists' interest
in the average case analysis of the minimum-bisection problem
led to intensive research on algorithms for recovering the partition~\cite{BCLS:87,B:87,DF:89,JS:98,CK:01,M:01}
(although not all of these works used the same model as us; for example,~\cite{BCLS:87} considered random regular graphs with a fixed minimum bisection size). At the same time,  the model is a classical statistical model for networks with communities~\cite{HLL:83}, and the questions of identifiability of the parameters and recovery of the clusters have been studied extensively, see e.g.~\cite{SN:97,BC:09,RCY:11}. We refer the readers to~\cite{MoNeSl:13} for more background on the model.

\subsection{The block models in sparse graphs}
Until recently, all of the theoretical literature on the stochastic block model focused on what we call the dense case, where the average degree is of order at least $\log n$ and the graph is connected. Indeed, it is clear that connectivity is required, if we wish to label all vertices accurately. However, the case of sparse graphs with constant average degree is well motivated from the perspective of real networks, see
e.g.~\cite{LLDM:08,Strogatz:01}.

Although sparse graphs are natural for modelling
many large networks, the stochastic block model
seems to be most difficult to analyze in the sparse
setting. Despite the large body of work about
this model, until recently the only result
for the sparse case $q, q' = O(\frac{1}{n})$
was that of Coja-Oghlan~\cite{CO:10}.
Recently, Decelle et al.~\cite{Z:2011} made some fascinating
conjectures for the cluster identification problem
in the sparse stochastic block model. In what follows,
we will set $q = a/n$ and $q' = b/n$ for some fixed $a, b$. It will be useful to parameterize these by $d=(a+b)/2$ and $s=(a-b)/2$.
Note that with these parameters, $s^2 > d$ implies that $s, d > 1$.

\begin{conjecture}\label{conj:reconstruction}
  If $s^2 > d$ then the clustering problem
  in $\calG(n, \frac an, \frac bn)$
  is solvable as $n \to \infty$, in the sense that one can
  a.a.s.\ find a bisection
  whose correlation with the planted bisection is bounded away from  0.
\end{conjecture}

Decelle et al.'s work is based on deep but non-rigorous ideas
from statistical physics. In order to identify the best bisection,
they use the sum-product algorithm (also known as belief propagation).
Using the cavity method, they argue that the algorithm should work,
a claim that is bolstered by compelling simulation results.
By contrast the best rigorous work by Coja-Oghlan~\cite{CO:10} showed that if $s^2 > C d \log d$ for a large 
constant $C$, then the spectral method solves the clustering problem.
(After the current article first appeared as a preprint, independent
works~\cite{GuedonVershynin:14,YunProutiere:14} gave simple algorithms that work when $s^2 > C d$.)

What makes Conjecture~\ref{conj:reconstruction} even more appealing
is the fact that if it is true, it represents a threshold for the solvability
of the clustering problem. Indeed it was conjectured in~\cite{Z:2011} and proved in~\cite{MoNeSl:13} that
  if $s^2 \leq d$ then the clustering problem
  in $\calG(n, \frac an, \frac bn)$
  problem is not solvable as $n \to \infty$.
It was further shown in~\cite{MoNeSl:13} that $s^2 = d$ represents the threshold for identifiability of the parameters $a$ and $b$
as conjectured by~\cite{Z:2011}.

The threshold $d=s^2$ can be understood both in terms of spin systems and in terms of random matrices.  It was first derived heuristically as a stability condition for the belief propagation algorithm~\cite{Z:2011}.
Around a typical vertex, the joint distribution of the graph labeled by the clusters is asymptotically (in the sense of local weak convergence) a Galton-Watson branching process labeled with the free Ising model.  The threshold $d=s^2$ corresponds to the extremality or reconstruction threshold for the Ising model, the point at which information on the spin at the root can be recovered over arbitrarily long distances. In~\cite{MoNeSl:13} this property was used to show the impossibility of reconstructing clusters when $s^2 \leq d$.

In~\cite{Krzakala_etal:13} the threshold was heuristically derived by considering the spectrum of the matrix of non-backtracking walks.  On an Erd\H{o}s-R\'enyi random graph the bulk spectrum of this matrix has radius $d^{1/2}$.  When $s^2>d$ there is a natural construction of an approximate eigenvector of eigenvalue $s$ which thus escapes from the bulk exactly at $d=s^2$.  The random matrix interpretation plays a central role in our anaylsis and is discussed further in Section~\ref{s:RMT}.

\subsection{Notation}
We write graphs as $G = (V, E)$, where $V$ is a vertex set and
$E$ is the set of edges. We write $v \sim w$ if $\{v, w\} \in E$.
If we need to speak about several graphs at the same time, we may
write $V(G)$ or $E(G)$ in order to be clear that we are referring
to the vertices (or edges) of the graph $G$.
If $\sigma$ is a labelling on $V$ and $U \subset V$, then we
write $\sigma_U$ for the restriction of $\sigma$ to $U$.

In order not to be overwhelmed with quantifiers, 
we make heavy use of the asymptotic notations $o, O, \omega$, and  $\Omega$,
including in the antecedent. For example, the statement that
``$a_n = O(b_n)$ implies that $c_n = O(d_n)$'' means that for every
$C_1 > 0$ there exists some $C_2 > 0$ such that $a_n \le C_1 b_n$
for all $n$ implies that $c_n \le C_2 d_n$ for all $n$.
Given a collection of sequences $(a_{v,n})$ depending on some other
parameter $v$, we say that they satisfy some asymptotics uniformly in $v$
if the hidden constants or rates of convergence do not depend on $v$.
For example, ``$a_{v,n} = o(b_n)$ uniformly in $v$'' means that
there is some sequence $c_n \to 0$ such that $a_{v,n} / b_n \le c_n$
for all $v$ and all $n$.

We write that a sequence of events holds \emph{asymptotically
almost surely} (or a.a.s.) if their probabilities converge to one.

\section{Our results}

Our main result is a proof of Conjecture~\ref{conj:reconstruction}:

\begin{theorem} \label{thm:main}
  If $s^2 > d$ then the clustering problem
  in $\calG(n, \frac an, \frac bn)$
  is solvable as $n \to \infty$, in the sense that one can
  a.a.s.\ find a bisection
  whose correlation with the planted bisection is bounded away from  0.
\end{theorem}

Our algorithm is also computationally efficient, and can be implemented in almost linear time $O(n \log^2 n)$.
We recently learned that Laurent Massouli\'e independently found a different proof of the conjecture~\cite{Massoulie:13}.

We note that our proof of Theorem~\ref{thm:main} actually gives slightly stronger results. First, $a$ and $b$
do not need to be fixed, but may grow slowly with $n$:
\begin{theorem}\label{thm:main-growing}
 If $a, b = n^{o(1/\log \log n)}$ and $s^2 / d \ge \lambda > 1$ for all $n$ then the clustering problem
 in $\calG(n, \frac an, \frac bn)$ is solvable as $n \to \infty$, in the sense of Theorem~\ref{thm:main}.
\end{theorem}

Moreover, although our proof of Theorem~\ref{thm:main} does not give particularly good bounds
for the size of the correlation, it does show that the correlation tends to 1 as $s^2/d$ grows.
\begin{theorem}\label{thm:main-accurate}
 If $a, b = n^{o(1/\log \log n)}$ and $s^2 / d \to \infty$ then the clustering problem
 in $\calG(n, \frac an, \frac bn)$ is solvable as $n \to \infty$, in the sense that one can
 a.a.s.\ find a bisection that agrees with the planted bisection up to an error of $o(n)$ vertices.
\end{theorem}

\subsection{Proof strategy}

It was conjectured in~\cite{Z:2011} that a popular algorithm, belief propagation initialized with i.i.d. uniform messages, detects communities all the way to the threshold. However, analysis of belief propagation with random initial messages is a difficult task.
Krzakala et al.~\cite{Krzakala_etal:13} argued that a novel and very efficient spectral algorithm based on a non-backtracking matrix should also detect communities all the way to the threshold.  Unfortunately we were unable to follow the path suggested in~\cite{Krzakala_etal:13} and our algorithm for detection is not a spectral algorithm. Still, our analysis is based on the non-backtracking walk and we show that it can be implemented using matrix powering. The algorithm has very good theoretical running time $O(n \log^2 n)$ but the constant in the $O$ needed for the proof that the algorithm is correct is very large, so the algorithm described is not nearly as efficient as the one in~\cite{Krzakala_etal:13} (nor have we implemented it).

\begin{definition}
  A \emph{path} is a sequence $u_0, \dots, u_k$ of vertices such that
  for all $i$, $u_i \ne u_{i-1}$. (Note
  that we do not require vertices in a path to be connected
  by an edge in any given graph; thus, it might be more standard
  -- but also rather longer --
  to use the term \emph{path in the complete graph} instead.)
  We write $E(\gamma)$ for the set of $\{u_{i-1}, u_i\}$ and
  $V(\gamma)$ for the set $\{u_0, \dots, u_k\}$.

  A \emph{non-backtracking path} is a path
  $u_0, \dots, u_k$ such that for every
  $0 \le i \le k-2$, $u_{i} \ne u_{i+2}$.

  A \emph{self-avoiding path} is a sequence of vertices
  $u_0, \dots, u_k$ that are all distinct.
\end{definition}

The basic intuition behind the proof is that we should be expecting a larger number of non-backtracking paths of a given length $k$ in the graph between two vertices $u$ and $v$ if they have the same label, while a smaller number of non-backtracking paths in the graph is expected if the nodes $u$ and $v$ have different labels.

Instead of working with the number of non-backtracking walks in the graph, it is more convenient to work with a rank one correction, where an edge is
represented by $1-d/n$ and a non-edge by $-d/n$.  With this alteration the expected weight of each edge is $0$.

\begin{definition}
  Let $W_e = 1_{\{e \in E(G)\}} - d/n$.
For a non-backtracking path $\gamma=u_0,\ldots,u_k$, let
\[
X_\gamma = \prod_{i=1}^k W_{(u_{i-1}, u_i)}.
\]
Let $\NB_{k,u,v}$ denote the set of non-backtracking paths of length $k$ from $u$ to $u'$ and let
\[
  N_{u,v}^{(k)} = \sum_{\gamma\in \NB_{k,u,v}} X_\gamma.
\]
Where $k$ is clear from the context, we will sometimes just write
$N_{u,v}$.
\end{definition}

Our basic method is to show that $N_{u,v}^{(k)}$ is correlated
with $\sigma_u \sigma_v$ for some $k \sim \log n$.
In order to do this, we would like
to compute the expectation and variance of the $N_{u,v}$.
There is an obstacle, however, which is that on some very rare event there are many more
paths in the graph than there should be; this event throws off
the expectation and variance of $N_{u,v}$.
For intuition, take $k = \lceil \alpha \log n \rceil$
for some large constant $\alpha$ (in order to make our method work,
we will need $\alpha$ arbitrarily large if $s^2$ is arbitrarily close to $d$).
As we will show later, $N_{u,v}$ is of the order $s^k/n$ with
high probability. However, the expectation of $N_{u,v}$ could
be much larger.
Indeed, the probability that an $m$-clique containing $u$ and $v$ will appear
is at least $n^{-m^2} = e^{-m^2 \log n}$.
On the event of its appearance,
there are at least $(m-2)^k = e^{(m-2) \alpha \log n}$ non-backtracking
paths of length $k$ from $u$ to $v$ that stay entirely within the graph;
each of these paths $\gamma$ has $X_\gamma \approx 1$.
If $\log (m-2) \geq 2 \log s$ and $\alpha \log (m-2) > 2m^2$
(which can be achieved by first taking $m$ large enough depending on $s$
and then taking $\alpha$ large enough depending on $m$),
then these paths contribute an expected weight of at least
\[
e^{\alpha \log (m-2) \log n - m^2 \log n} \ge e^{\frac 12 \alpha \log(m-2) \log n}
\ge s^{\alpha \log n} = s^{k},
\]
which is of a larger order than $s^k/n$.
For this reason, our argument for controlling $N_{u,v}$ will be
divided into two parts: we will use the second moment method to control
the part of $N_{u,v}$ that involves ``nice'' paths, and we will
control the other paths by conditioning
on an event that excludes cliques, along with some other problematic structures which we call tangles. 

\begin{definition}
  An \emph{$\ell$-tangle} is a graph of diameter at most $2\ell$ that
  contains two cycles. We say that a graph is \emph{$\ell$-tangle-free} if
  if does not contain any $\ell$-tangles as subgraphs;
  i.e., every neighborhood of
  radius $\ell$ in the graph has at most one cycle.
\end{definition}

Standard random graph arguments imply that sparse graphs are $\ell$-tangle free with high probability,
for some $\ell$ depending on the sparsity. Indeed, we will show later
(Lemma~\ref{lem:xi-small}) that with the following choice of
parameters
(which we fix for the rest of this article),
$G$ is $\ell$-tangle-free with probability $1 - n^{-1 + o(1)}$:

\begin{assumption}\label{ass}
 Assume that $s, d = n^{o(1/\log \log n)}$, and fix a sequence $\ell = \ell_n$
 satisfying $\log \log n \ll \ell \ll \sqrt{\log n}$.
\end{assumption}

Roughly speaking, our main technical result is that
$N_{u,v}$ is correlated with $\sigma_u \sigma_v$, and
that as $u$ and $v$ vary then the variables $N_{u,v}$ are
essentially uncorrelated.

\begin{theorem}\label{thm:path-weights-main}
  Assuming that $s^2/d \ge \lambda > 1$, choose $\alpha$ so that
  $n^2 d^{\alpha \log n} \le s^{2\alpha \log n}$ for every $n$.
  Let $u, v, u', v'$ be distinct vertices, and let
  $\SAW_{u,v}$ be the set of self-avoiding paths from 
  $u$ to $v$ of length $k = \lceil \alpha \log n \rceil$.
  Suppose that $U, U' \subset V$ contain $\{u, v\}$ and $\{u', v'\}$
  respectively, and that both have cardinality at most $n^{o(1)}$.
  Let $Y_{u,v} = \sum_{\gamma \in \SAW_{u,v}} X_\gamma$.
  Then, uniformly over $u, v, u', v'$, and all labellings
  $\sigma_U, \sigma_{U'}$ on $U$ and $U'$,
  \begin{align}
    \label{eq:thm-first-moment}
    \E \left[ Y_{u,v} \mid \sigma_U \right]
    &= (1 + n^{-1+o(1)}) \frac{\sigma_u \sigma_v s^k}{n} \\
    \label{eq:thm-second-moment}
    \E \left[Y_{u,v}^2 \mid \sigma_U \right] 
    &\le (1 + o(1)) 2\left(\frac{s^2}{s^2-d}\right) \frac{s^{2k}}{n^2} \\
    \label{eq:thm-cross-moments}
    \E \left[Y_{u,v} Y_{u',v'} \mid \sigma_U \sigma_{U'} \right]
    &= (1 + n^{-1 + o(1)}) \E [Y_{u,v} \mid \sigma_U, \sigma_{U'}] \E [Y_{u',v'} \mid \sigma_U, \sigma_{U'}] \\
    \label{eq:thm-bad-paths}
\P \left[ |N^{(k)}_{u,v} - Y_{u,v}| \ge s^k n^{-4/3} \mid \sigma_U \right]
&\le n^{-1/3 + o(1)}.
  \end{align}
\end{theorem}

It follows easily from Theorem~\ref{thm:path-weights-main} that if $s^2/d \to \infty$
then we can get very accurate estimates of $\sigma_u \sigma_v$ by
computing $N_{u,v}^{(k)}$. To achieve
non-trivial estimates of $\sigma_u \sigma_v$ in the case
$s^2/d > 1$ is more complicated. We will explain the procedure roughly
in the next section.

\subsection{Almost linear time algorithm}

Theorem~\ref{thm:path-weights-main} suggests a natural way to check if two vertices are in the same cluster;
this is the basis of the algorithm we develop to cluster the graph. We further show how to efficiently perform the algorithm using matrix powering.

\begin{theorem} \label{thm:alg}
  There is an algorithm that runs in time $O(nd\log^2 n)$ and satisfies
  the following guarantee:
  for any $\lambda > 1$ there is an $\epsilon > 0$ such that
  if $s^2 / d \ge \lambda > 1$ for all $n$ then the algorithm,
  given $G \sim \calG(n, \frac an, \frac bn)$,
  produces a labelling $\tau$ satisfying
  \[
    \left|\frac 1n \sum_v \sigma_v \tau_v \right| \ge \epsilon > 0,
  \]
  with probability $1-o(1)$,
  where $\sigma$ is the true labelling of $G$.
\end{theorem}
With more care in the analysis the running time could be reduced to $O(n d \log n)$
for a slightly modified algorithm.

\subsection{Connections with random matrix theory}\label{s:RMT}

Let $A$ be the adjacency matrix of a graph $G$ from
$\calG(n, \frac an, \frac bn)$. One standard spectral algorithm for community
detection takes
the top eigenvector of $A - \frac dn \1$ (where $\1$ is the $n \times n$
matrix filled with ones) and rounds it to reconstruct approximations of the clusters. To analyze
this algorithm (see, e.g.~\cite{NadakuditiNewman:12}), consider the labelling
$\sigma$ as an vector in $\{1, -1\}^n$; then conditioned on $\sigma$,
the random matrix $M = A - \frac dn \1 - \frac sn \sigma \sigma^T$ has
independent, zero-mean entries.
If $a$ and $b$ are growing not too slowly with $n$ then one can show a
Bai-Yin theorem
for $M$, and it follows that the top eigenvector of the rank-one perturbation
$M + \frac sn \sigma\sigma^T
= A - \frac dn \1$ is correlated with the true labelling $\sigma$.
The random matrix aspects of this analysis have received substantial
attention in recent years. For example, semi-circle laws and local statistics
are now known whenever $a$ and $b$ are growing at least logarithmically
fast in $n$~\cite{Wood:12} (following earlier work that required polynomially fast growth~\cite{EKYY:12,TaoVu:08}), and also for matrix ensembles
with more general i.i.d.\ entries. These ideas were used in~\cite{NadakuditiNewman:12} to establish the block model threshold conjecture in the case where $a$ and $b$ grow at least logarithmically in $n$.

The preceding arguments -- and also the more general random matrix theory --
break down in the sparse case, where $a$ and $b$
are $O(1)$. One reason for this is the presence of high-degree vertices: with
high probability there exist vertices with degree $\Omega(\log n / \log \log n)$, and these
play havoc with the spectrum of $A$.
We emphasize that this is not merely a looseness in the analysis: naive spectral
algorithms genuinely fail for sparse graphs, see e.g.~\cite{Krzakala_etal:13}
for a discussion of why this happens for the stochastic block model,
or~\cite{FlFrFe:05} for a quite different example of the connection
between vertex degree and spectrum in random graphs.

Some attempts were made to modify spectral methods.
A popular idea is to prune all nodes whose degree is larger than some big constant.
This idea was pioneered by Feige and Ofek~\cite{FeigeOfek:05} for a
different application, and studied in our context by
Coja-Oghlan~\cite{CO:10},
who gave a spectral algorithm that succeeds on sparse graphs but not
all the way to the threshold: it requires $s^2 > C d \log d$ for
some constant $C$.
  
Krzakala et al.~\cite{Krzakala_etal:13} suggest quite a different
way to ``fix'' the spectrum of $A$: instead of $A$, they consider
a non-symmetric matrix that avoids the contribution of high-degree vertices
by counting non-backtracking paths in the graph instead of all paths in the graph. 
Although simulations strongly suggested that the non-backtracking matrix
had desirable spectral properties whenever $s^2 > d$, the proof
remained elusive until very recently
(and after the first appearance of this work),
when Bordenave et al.~\cite{BoLeMa:15} gave a solution.
Their result required new developments in random matrix theory,
partly because the non-backtracking matrix is very sparse and partly
because its entries are far from i.i.d.
These methods were further developed by Bordenave~\cite{Bordenave:15},
who gave a new proof of Alon's conjecture for the second eigenvalue
of random $d$-regular graphs.

In an independent (and concurrent) work, Massouli\'e~\cite{Massoulie:13} gave a
proof of Theorem~\ref{thm:main} using a spectral algorithm. He considered
the matrix $M$ where $M_{uv}$ is the number of self-avoiding walks between
$u$ and $v$ of length $\alpha \log n$, for some not-too-large constant $\alpha$.
This ``regularized'' matrix has several advantages over the non backtracking matrix we consider.
The matrix is quite dense (all degree are polynomial) and in fact is close to regular. Moreover, the matrix is symmetric which allows standard perturbation theory to apply. On the other hand, the entries of the matrix are not independent.
Still, Massouli\'e showed how to apply the trace method and analyze the spectrum of the matrix.
He proved that if $s^2 > d$ then $M$ has
a separation
between its second- and third-largest eigenvalues, and that the second eigenvector
is correlated with the true labelling. Hence, the spectral algorithm that
rounds the second eigenvector of $M$ succeeds down to the threshold.
We note that the algorithm we describe is much more efficient than the one by~\cite{Massoulie:13}, which might not even be implementable in polynomial time
(for example, counting self-avoiding walks is \#P-complete~\cite{LiOgTo:03} even for fairly simple families of graphs); in any case, simply writing down the dense $n \times n$ matrix in~\cite{Massoulie:13} will take time $O(n^2)$.

\section{The algorithm and its running time}

In this section, we will describe the algorithm and give
its analysis assuming Theorem~\ref{thm:path-weights-main}.
We will begin by describing how to use the quantities $N_{u,v}^{(k)}$
to estimate the graph labelling. In
Section~\ref{s:NCalculation}, we will show how these quantities
may be computed efficiently, thereby completing the description
of our algorithm. In Section~\ref{sec:alg-analysis}, we prove the algorithm's
correctness.

\subsection{The algorithm}
Recall that our random graph model adds a within-class edge
with probability $a/n$ and a between-class edge with probability $b/n$,
where $a$ and $b$ are parameters that may grow slowly with $n$.
We set $d = (a+b)/2$ and $s = (a-b)/2$, and assume that $s^2 / d \ge \lambda
> 1$ for all $n$.

We begin by describing a simplified version of our algorithm.
This simplified version runs more slowly, but
it is more intuitive and will serve to motivate the main algorithm.
The basic idea is to fix a very slowly increasing sequence, say
$R = R_n = 2 \lceil \log \log \log \log n \rceil$.
We write $B_r(v)$ for the set of vertices whose path distance to $v$
in $G$ is at most $r$, and we write $S_v = B_R(v) \setminus B_{R-1}(v)$.
Fix a node $w^*$ with large degree (at least $\sqrt{\log \log n}$, say).
For every other node $v$, consider the graph $H = H(v)$
obtained by removing $w^*$ and $B_{R-1}(v)$ from $G$. Our estimate
for $\sigma_v$ will be
\[
  \tau_v = \sgn \left(\sum_{u \sim w^*} \sum_{w \in S_v} N_{u,w}^{(k)}\right),
\]
where $k = \Theta(\log n)$ and $N_{u,w}^{(k)}$
is computed with respect to the graph $H$.
After observing that $H$ is essentially distributed according to the
stochastic block model,
Theorem~\ref{thm:path-weights-main} and the fact (coming from the
theory of multi-type branching processes) that $\sum_{w \in S_v} \sigma_w$
is typically of a larger order than $\sqrt{|S_v|}$ together imply
that $\tau_v = \sgn \sum_{u \sim w^*, w \in S_v} \sigma_u \sigma_w$
with probability going to 1. The theory of branching processes
also implies that when $s^2 > d$ then the sign of $\sum_{w \in S_v}$ is
non-trivially correlated with $\sigma_v$. Hence, $\tau$ is non-trivially
correlated with $\sigma$.

The preceding algorithm has two flaws that we will correct shortly.
First, the distribution of $H$ is slightly painful to work with,
because after removing the node $w^*$ the remaining edges are no longer
independent.
Second, the running time of the algorithm above will be about
$O(n^2 \log n)$, because
we must compute the numbers $N_{u,w}^{(k)}$ (each of which takes time
$O(n \log n)$) with respect to $O(n)$ different graphs $H(v)$. This
could be fixed by handling several nodes simultaneously: we could
remove $\bigcup_v B_{R-1}(v)$ from $G$, where the union is taken over,
say, $n/\log n$ vertices $v$. This is almost the approach that we will
take, but we need to be careful that whatever we remove, the
remaining graph is almost distributed according to the
stochastic block model. We will ensure this by a slightly convoluted
plan: instead of removing specific nodes
and neighborhoods, we will remove $\delta n$
vertices from $G$ uniformly at random and look for nodes and neighborhoods
that are contained in the removed part. The precise description of
our algorithm follows:

Let $R = 2 \lceil \log \log \log \log n \rceil$.
Let $\delta' >0 $ be chosen so that $s^2 (1-\delta')^2 =  d (1-\delta')$ and let $\delta = \delta'/2$. We will choose constants $\kappa$ and $k$ depending
on $s'$ and $d'$ (the precise dependence will be given later).
Then, we proceed as follows:

\begin{itemize}
\item
Remove at random $\lceil \sqrt{n}\rceil$ vertices $V''$ from the graph , leaving the graph $G'=(V',E')$
\item
Let $w^{\ast}$ be a node in $V''$ whose number of neighbors in $V'$ is closest to $\lceil\sqrt{\log\log n}\rceil$.  Let $S_\ast$ be the set of its neighbors in $V'$.
\item
For each $v\in V'$ denote $S_v=B_{R}(v) \setminus B_{R-1}(v)$.
\item
For each $1 \leq j \leq \log n$, let $U_j$ be a uniformly random set
of $\lceil n \delta\rceil - \lceil\sqrt{n}\rceil$ vertices of $V' \setminus S_\ast$; set
$V_j = V' \setminus (S_\ast \cup U_j)$.
\item For each $v\in V'$ and $1 \leq j \leq \log n$ define
\[
  \tau_{j,v} = \sgn\left(\sum_{u \in S_\ast \cap V_j, u' \in S_v\cap V_j}  N^{(k,j)}_{u,u'} + \kappa \frac{{s'}^{k+R+1}}{d' }  |S_{\ast}|  \xi_{j,v}\right),
\]
where $\xi_{j,v}$ are i.i.d.\ random variables uniform on $[-1,1]$ and
\[
  N^{(k,j)}_{u,u'} = \sum_{\stackrel{\gamma\in \NB_{k,u,u'}}{\gamma\subset V_j}} X_{\gamma,d'}.
\]
Here, we define the function $\sgn$ so that $\sgn(0) = 0$.
\item
  For each $v\in V'$ let $J_v$ be the first $j$ such that $B_{R-1}(v)\cap V_j = \emptyset$ and $(S_v \cup S_\ast) \subset V_j$, and 0 if no such $j$ exists.  Then set $\tau(v) = \tau_{J_v,v}$ when $J_v\neq 0$ and $\tau_{J_v,v} \ne 0$.
  For all other $v \in V$, choose $\tau(v)$ at independently at random uniformly from $\{1,-1\}$.
\end{itemize}

We will prove Theorem~\ref{thm:alg} in Section~\ref{sec:alg-analysis}
by showing that
the output $\tau$ of the algorithm above is correlated with the true partition with high probability.
In the following section  we describe how to evaluate the $\tau$ in time $O(nd \log^2 n)$.

\subsection{Efficient implementation of the algorithm}\label{s:NCalculation}
The main computational step in the algorithm above is to compute
$N^{(k,j)}_{u,u'}$; in this section, we will describe how to do so.
First, however, note that $N^{(k,j)}_{u,u'}$ is just $N^{(k)}_{u,u'}$
computed on the subgraph induced by $V' \setminus (S_\ast \cup V_j)$.
In particular, it is enough to show how to compute $N^{(k)}_{u,u'}$
efficiently.

Let $V$ be the  vertex set and $A$  the adjacency matrix of the graph $G$.
We recall the definition of $N^{(k)}_{u,v}$ and introduce some related matrices

\begin{definition} Let $I$ denote the $n \times n$ identity matrix,
  let $D$ be the $n \times n$ diagonal matrix whose $u$th diagonal 
  entry is the degree of the vertex $u$, and let
  $\1$ be the $n \times n$ matrix all of whose entries are $1$.
  Define the $n \times n$ matrix $N^{(k)}$ by
\[
  N^{(k)}_{u,v} = \sum_{\gamma \in \NB_{k,u,v}} X_\gamma,
\]
with the convention that $N^{(0)}=I$. For $\rho \in \R$, define the $n \times n$ matrix $Q^{(k,\rho)}$ by
\begin{equation} \label{eq:Q}
Q^{(k,\rho)}_{u,v} = \sum_{j=0}^{\lfloor k/2 \rfloor} \rho^{2j} N^{(k-2j)}_{u,v}
\end{equation}
(where an empty sum is defined to be zero, so $Q^{(k,\rho)}$
is the zero matrix for $k < 0$),
and define the $4n \times 4n$ matrices
\begin{equation} \label{eq:M}
\calM= \left(
     \begin{array}{cccc}
       (1-d/n) A & -(1-d/n)^2 (D-I) & - (d/n)(\mathbbm{1}-A-I) & - (d/n)^2 ((n-1)I-D) \\
       I & 0 & 0 & 0 \\
       (1-d/n) A & -(1-d/n)^2 D & - (d/n)(\mathbbm{1}-A-I) & - (d/n)^2 ((n-2)I-D) \\
       0 & 0 & I & 0 \\
     \end{array}
   \right),
\end{equation}
\begin{equation} \label{eq:hatM}
\hat{\calM}= \left(
     \begin{array}{cccc}
       (1-d/n) A & -(1-d/n)^2 D & - (d/n)(\mathbbm{1}-A-I) & - (d/n)^2 ((n-1)I-D) \\
       0 & 0 & 0 & 0 \\
       0 & 0 & 0 & 0 \\
       0 & 0 & 0 & 0 \\
     \end{array}
   \right).
\end{equation}
Finally, define the $4n \times n$ matrix $\mathcal{Q}_k$ by
\begin{equation} \label{eq:mathQ}
\mathcal{Q}_k= \left(
               \begin{array}{cccc}
                 Q^{(k,1-d/n)} & Q^{(k-1,1-d/n)} & Q^{(k,-d/n)} & Q^{(k-1,-d/n)} \\
               \end{array}
             \right)^T
\end{equation}
\end{definition}

\begin{lemma} \label{lem:matrices}
We have
\[
\calM \mathcal{Q}_k  = \mathcal{Q}_{k+1}
\]
and
\[
\hat{\calM} \mathcal{Q}_k= \left(
               \begin{array}{cccc}
                 N^{(k+1)} & 0 & 0 & 0 \\
               \end{array}
             \right)^T.
\]
\end{lemma}

\begin{proposition} \label{prop:efficient}
For a graph on $n$ vertices and $m$ edges and for every vector $z$ the matrix
$N^{(k)} z$ can be computed in time $O((m+n) k)$.
\end{proposition}

\begin{proof}
The proof follows from the fact that
(by Lemma~\ref{lem:matrices}) $N^{(k)} z$ 
is the first $n$ coordinates of
\[
\hat{\calM} \calM^{k-1} \mathcal{Q}_0 z,
\]
and that each of the matrices $\calM, \hat{\calM}$ and $\mathcal{Q}_0$ is made of at most $16$ blocks, each of which is a sum
of a sparse matrix with $O(n+m)$ entries and a rank $1$ matrix.
Therefore, the displayed expression above can be computed with
$k+1$ matrix-vector multiplications, each of which requires 
$O(n+m)$ time.
\end{proof}

We can now prove the running time bound in Theorem~\ref{thm:alg}.
Indeed, in iteration $j$ of the algorithm,
the sum $\sum_{u \in S_\ast \cap V_j, u' \in S_v \cap V_j}  N_{u,u'}^{(k,j)} $ is the non-trivial computation that needs to be done.
This sum can be read from the entries of $N^{(k)} z$, where $N$ is computed on the graph with the removed nodes
and $z$ is the indicator of the vertices in $S_\ast$. 
By Proposition~\ref{prop:efficient}, the running time of iteration $j$ is $O((n+m)k)$. Since there are $\log n$ iterations and we have $m=O(nd)$ and $k=O(\log n)$ we obtain that the running time of the algorithm is $O(nd \log^2 n)$.

It remains to prove Lemma~\ref{lem:matrices}.
\begin{proof}[Proof of Lemma~\ref{lem:matrices}]
  We will write $N^{(k)}_{u,w,v}$ to denote the $N^{(k)}_{u,v}$, but with
  the sum restricted to non-backtracking paths which move to
 $w$ on their first step. Then we have the recursion
\begin{align*}
N^{(k)}_{u,v} &= \sum_{w\neq u}  N^{(k)}_{u,w,v} \\
&=  \sum_{w\neq u} W_{\{u,w\}} (N^{(k-1)}_{w,v} - N^{(k-1)}_{w,u,v})\\
&=  \sum_{w\neq u} W_{\{u,w\}} \left(N^{(k-1)}_{w,v} - W_{\{u,w\}} (N^{(k-2)}_{u,v} - N^{(k-2)}_{u,w,v})\right)\\
& = \sum_{w\neq u} W_{\{u,w\}} N^{(k-1)}_{w,v} -  \sum_{w\neq u} W_{\{u,w\}}^2 N^{(k-2)}_{u,v} +
\sum_{w \neq u} W_{\{u,w\}}^2  N^{(k-2)}_{u,w,v}
\end{align*}
By expanding the terms of the form $N^{(k-2)}_{u,w,v}$ repeatedly, we obtain by induction that
\begin{align*}
N^{(k)}_{u,v} &=  \sum_{w\neq u}\Bigg[ \Big(\sum_{j=0}^{\lfloor (k-1)/2 \rfloor} W_{\{u,w\}}^{1+2j} N^{(k-1-2j)}_{w,v} \Big) -  \Big(\sum_{j=0}^{\lfloor (k-2)/2 \rfloor} W_{\{u,w\}}^{2+2j} N^{(k-2-2j)}_{u,v}\Big) \bigg]\\
&=  \sum_{w\sim u}\Bigg[ \Big(\sum_{j=0}^{\lfloor (k-1)/2 \rfloor} (1-d/n)^{1+2j} N^{(k-1-2j)}_{w,v} \Big) -  \Big(\sum_{j=0}^{\lfloor (k-2)/2 \rfloor} (1-d/n)^{2+2j} N^{(k-2-2j)}_{u,v}\Big) \bigg]\\
&+ \sum_{\substack{w\not\sim u\\w\neq u}}\Bigg[ \Big(\sum_{j=0}^{\lfloor (k-1)/2 \rfloor} (-d/n)^{1+2j} N^{(k-1-2j)}_{w,v} \Big) -  \Big(\sum_{j=0}^{\lfloor (k-2)/2 \rfloor} (-d/n)^{2+2j} N^{(k-2-2j)}_{u,v}\Big) \bigg]
\end{align*}
 The recursion above can be written using the matrix $Q$ from~(\ref{eq:Q}) in the following way:
\begin{align*}
N^{(k)} &=  (1-d/n) A Q^{(k-1,1-d/n)} -(1-d/n)^2 D Q^{(k-2,1-d/n)}\\
&\qquad - (d/n)(\mathbbm{1}-A-I) Q^{(k-1,-d/n)} - (d/n)^2 ((n-1)I-D) Q^{(k-2,-d/n)}
\end{align*}
Moreover,
\begin{align*}
Q^{(k,1-d/n)}&= N^{(k)} +  (1-d/n)^2 Q^{(k-2,1-d/n)} \\
 &=  (1-d/n) A Q^{(k-1,1-d/n)} -(1-d/n)^2 (D-I) Q^{(k-2,1-d/n)}\\
&\qquad - (d/n)(\mathbbm{1}-A-I) Q^{(k-1,-d/n)} - (d/n)^2 ((n-1)I-D) Q^{(k-2,-d/n)},
\end{align*}
and
\begin{align*}
Q^{(k,-d/n)}&= N^{(k)} +  (d/n)^2 Q^{(k-2,-d/n)} \\
 &=  (1-d/n) A Q^{(k-1,1-d/n)} -(1-d/n)^2 D Q^{(k-2,1-d/n)}\\
&\qquad - (d/n)(\mathbbm{1}-A-I) Q^{(k-1,-d/n)} - (d/n)^2 ((n-2)I-D) Q^{(k-2,-d/n)}.
\end{align*}
Written in terms of the matrices $\mathcal{Q}$ from~(\ref{eq:mathQ}) and $\calM,\hat{\calM}$ defined in~\eqref{eq:M} and~\eqref{eq:hatM}, the recursions above can be written as
$
\calM \mathcal{Q}_k  = \mathcal{Q}_{k+1}
$
and
$
\hat{\calM} \mathcal{Q}_k= \left(
               \begin{array}{cccc}
                 N^{(k+1)} & 0 & 0 & 0 \\
               \end{array}
             \right)^T,
$
as claimed.
\end{proof}

\subsection{Correctness of the algorithm}\label{sec:alg-analysis}

In this section, we will prove the correctness of the algorithm
assuming Theorem~\ref{thm:path-weights-main}. We begin with some
preliminary observations about the distributions of various subgraphs
of $G$.
The distribution of $G'$ is simply a stochastic block model with fewer vertices, that is $G(n - \lceil \sqrt{n} \rceil,a/n,b/n)$.
Let $G_j = (V_j,E_j)$ denote the graph obtained at iteration $j$;
$G_j$ is also distributed as a stochastic block model, but we will need
to say more because we will need to use $G_j$ conditioned on some extra properties. In particular, we need to argue that conditioned on a vertex neighborhood being removed, the distribution on the remaining graph is drawn (approximately) from the block model. The technical issue here is that the removed vertices are correlated and moreover we need to condition on some of their labels.
Nevertheless, this can be handled because the neighborhood of a single vertex does not contain too many other vertices. 

For a vertex $v$ in $G_j$, let $U = U(v)$ denote the set $S_v \cup S_{\ast}$.
We will be interested in the distribution of $(G_j,U,\sigma_U)$ and we would like to couple it with a configuration of
$(G',U',\sigma'_{U'})$ drawn from $\calG(n-\lceil \delta n \rceil,a/n,b/n)$ and $U'$ is some fixed set of vertices of size $|U|$.

\begin{lemma} \label{lem:Gj_cond}
Fix a vertex $v \in V'$ and a labelling $\tau$ of $V$.
Let $\P_1$ denote the distribution of $(G_j, \sigma_{V_j})$ conditioned on $S_v \cup S_\ast$, the graph structure of $B_R(v)$ and the events
 \[
   |B_R(v)| \leq n^{0.1}, \quad B_{R-1}(v) \cap V_j = \emptyset, \quad S_v \subset V_j, \quad \sigma_U = \tau_U, \quad \sigma_{B_{R-1}(v)} = \tau_{B_{R-1}(v)},
\]
where $U = S_v \cup S_\ast$.

Given a set $U'$ of vertices and a labelling $\tau'$,
let $\P_2$ denote the distribution of
$(G', \sigma') \sim \calG(n - \lceil \delta n\rceil, a/n, b/n)$
conditioned on $\sigma'_{U'} = \tau'_{U'}$.
Suppose that we can identify $V(G')$ and $V(G)$ in such a way that
$U = U'$ and $\tau$ agrees with $\tau'$ on $U$.
Then for large enough $n$, the measures $\P_1$ and $\P_2$ satisfy that
\[
d_{TV}(\P_1,\P_2) \leq n^{-0.3}.
\]
\end{lemma}

\begin{proof}
  The proof will couple $\sigma'$ with $\sigma_{V_j}$ and the edges
  of $G'$ with the edges of $G_j$.
The coupling proceeds in the following way:
\begin{itemize}
\item
  We take $\sigma'$ and $\sigma$ to be equal on $U = U'$
  (neither one is random in either measure).
\item
Then we try to couple all other labels so they are completely identical.
\item
Finally, if the labels are identical, we will include exactly the same edges. This is possible since different edges are independent
and the probabilities of including edges just depend on the end points.
\end{itemize}

The only non-trivial part of this proof is showing that we can perform the second step with high probability.
Note that in $\P_2$, all of the labels outside $U'$ are independent
and uniformly distributed.
The conditional distribution outside $U$ in $G_j$ under the conditioning is also i.i.d.\ (since no edges are revealed).
However, in $\P_1$ the labels are biased, since we know they were not connected to the vertices of $B_{R-1}(v)$.
Indeed, for each vertex $u$ outside $U$ we have that
\[
\frac{\P_1[\sigma_u = +]}{\P_1[\sigma_u = -]} = \left(\frac{1-\frac{a}{n}}{1-\frac{b}{n}}\right)^{n_+-n_-},
\]
where $n_{\pm}$ is the number of $\pm 1$ labels in $\sigma_{B_{R-1}(v)}$.
Thus
\[
\P_1[\sigma_u = +] = \frac{1}{2} + O(|B(v,R)|/n) = \frac{1}{2} + O(n^{-0.9}).
\]
It is well known (see e.g.~\cite{Roos:2001}) that
\[
d_{TV}(\Binom(n,1/2),\Binom(n,1/2+x)) = O(x \sqrt{n})
\]
which therefore implies that
$
d_{TV}(\P_1,\P_2) \leq O(n^{-0.4}) \leq n^{-0.3},
$
as claimed.
\end{proof}

Next we note that with high probability $|S_\ast|=\lceil\sqrt{\log\log n}\rceil$ since the probability that there exists a vertex in $V''$ with that number of neighbors tends to one; we will condition on this event.
Also, we may assume without loss of generality that $\sigma_{w^{\ast}} = +$.
We will denote 
\begin{equation}\label{eq:M-def}
  M_v = \sum_{v\in S_v} \sigma_v \text{ and } M_\ast = \sum_{v\in S_{\ast}} \sigma_v.
\end{equation}
Note that $M_\ast$ is a sum
of i.i.d. signs, each of which has expectation $\frac sp$ (since we conditioned on $\sigma_{w^\ast} = +$).
Hence, a.a.s.
\[
  |M_\ast -\frac{s}{d} |S_{\ast}| | \leq |S_{\ast}|^{3/4},
\]
which means in particular that $M_\ast$ a.a.s.\ has the same sign as
$s$.

Before we proceed to the estimates that apply specifically for our
algorithm, let us note a simple corollary of the first three
statements of Theorem~\ref{thm:path-weights-main}:
\begin{lemma}\label{lem:self-avoiding}
  Take disjoint sets $U_1, U_2 \subset V$ that have cardinality
  $n^{o(1)}$; let $U = U_1 \cup U_2$. Under the notation and assumptions of
  Theorem~\ref{thm:path-weights-main}, if $Y = \sum_{u \in U_1, v \in U_2} Y_{u,v}$ then uniformly for all labellings $\sigma_U$ on $U$,
  \begin{align*}
    \E \left[Y \mid \sigma_{U}\right]
  &= (1 + n^{-1 + o(1)}) \frac{s^k}{n} \sum_{u \in U_1,v \in U_2} \sigma_u \sigma_v \\
    \Var \left[Y \mid \sigma_{U}\right]
    &= O\left(|U_1| |U_2| (|U_1| + |U_2|) \frac{s^{2k}}{n^2}\right).
  \end{align*}
\end{lemma}

\begin{proof}
  By Theorem~\ref{thm:path-weights-main},
  \[
    \E \left[Y \mid \sigma_{U}\right]
    = \sum_{u \in U_1, v \in U_2} \E \left[Y_{u,v} \mid \sigma_{U}\right]
    = (1 + n^{-1 + o(1)}) \frac{s^k}{n}
    \sum_{u \in U_1, v \in U_2} \sigma_u \sigma_v,
  \]
  as claimed.

  For the second moment,
  \[
    \E \left[Y^2 \mid \sigma_{U}\right]
    = \sum_{u,u' \in U_1, v, v' \in U_2} \E\left[Y_{u,v} Y_{u', v'} \mid \sigma_{U}\right].
  \]
  We divide the sum into three parts: the first part
  (containing $|U_1| |U_2|$ terms) sums over $u = u'$ and $v = v'$;
  for this part,
  we apply~\eqref{eq:thm-second-moment}.
  The second part
  (containing less than $|U_1|^2 |U_2| + |U_2|^2 |U_1|$ terms)
  sums over indices with either $u = u'$ or $v = v'$; for this
  part, we use the bound
  \[
    \E [Y_{u,v} Y_{u',v'} \mid \sigma_{U}]
    \le \E [Y_{u,v}^2 \mid \sigma_{U}]^{1/2} \E [Y_{u',v'}^2 \mid \sigma_{U}]^{1/2}
  \]
  and then apply~\eqref{eq:thm-second-moment} to each term on the right hand side.
  Finally, the third part ranges over distinct $u, u', v, v'$
  (less than $|U_1|^2 |V_1|^2$ terms), and we apply~\eqref{eq:thm-cross-moments}
  Putting these three parts together,
  \begin{multline*}
    \E \left[Y^2 \mid \sigma_{U}\right]
    \le (1 + o(n^{-1+o(1)})) \Big[(|U_1| |U_2| + |U_1|^2 |U_2| + |U_1| |U_2|^2)
    2\left(\frac{s^2}{s^2-d}\right) \frac{s^{2k}}{n^2}  \\
    + \frac{s^{2k}}{n^2} \sum_{u \ne u'} \sum_{v \ne v'} \sigma_u \sigma_v \sigma_{u'} \sigma_{v'} 
  \Big].
  \end{multline*}
  Finally, note that the second term above differs from $(\E [Y \mid \sigma_{U}])^2$ by at most $2 |U_1| |U_2| (|U_1| + |U_2|) s^{2k}/n^2$.
  Subtracting $(\E [Y \mid \sigma_{U}])^2$ from the displayed equation above
  thus proves our claim about the variance of $Y$.
\end{proof}

We may apply the previous lemma with~\eqref{eq:thm-bad-paths} and
Chebyshev's inequality to show that 
\[
  Z_v := \sum_{u \in S_\ast \cap V_{J_v}, u' \in S_v \cap V_{J_v}} N_{u,u'}^{(k,J_v)}
\]
can be used to
estimate the sign of $M_v$ (which, recall, was defined
in~\eqref{eq:M-def}).

\begin{lemma}\label{l:MainN}
For a random vertex $v$ and any $\epsilon>0$, conditioned on $J_v \ne 0$,
\[
  \P\left[\left|Z_v - \frac{{s'}^{k} }{n}M_z M_\ast \right| > \frac{\epsilon {s'}^{k +R}}{n} |S_{\ast}|  \right] \to 0 \text{ as $n \to \infty$.}
\]
\end{lemma}

\begin{proof}
  We condition on $J_v = j \ne 0$ and work with the measure $\P_2$ from
  Lemma~\ref{lem:Gj_cond}; we will also condition on the
  (high probability) event that $|S_v|^2 \le |S_\ast| = n^{o(1)}$.
  Hence, we can apply Lemma~\ref{lem:self-avoiding} to the graph $G_j$
  with $U_1 = S_v$, $U_2 = S_{\ast}$, and with slightly
  different graph parameters:
  $n - \lceil \delta n\rceil$ is the number of nodes, and the parameters
  $s$ and $d$ are replaced by $s'$ and $d'$. Setting
  $Y_v = \sum_{u \in S_v, u' \in S_\ast} Y_{u,u'}$ (where $Y_{u,u'}$ is now computed
  with respect to the graph $G_j$), Lemma~\ref{lem:self-avoiding}
  and Chebyshev's inequality give that for any $t \ge 1$,
  \[
    \P_2 \left[
      \left|
        Y_v - \frac{{s'}^k}{n-\delta n} M_z M_*
      \right| > \frac{2t {s'}^k}{n - \delta n} |S_{\ast}|
    \right] \le O(t^{-2} (|S_v| + |S_v|^2/|S_{\ast}|)).
  \]

  Next, we control $Z_v - Y_v$. By~\eqref{eq:thm-bad-paths} and
  a union bound,
  \[
    \P_2 \left[
      |Z_v - Y_v| \ge \frac{t {s'}^k}{n - \delta n} |S_{\ast}|
    \right]
    \le \sum_{u \in S_v, u'\in S_{\ast}} \P_2 \left[
      |Z_{u,u'} - Y_{u,u'}| \ge \frac{t {s'}^k}{|S_v|(n - \delta n)}
    \right].
  \]
  Since $|S_v|$ and $|S_\ast|$ are $n^{o(1)}$,~\eqref{eq:thm-bad-paths}
  implies that for any $t \ge 1$, the right hand side above converges
  to zero. Putting it together and setting $t = \epsilon {s'}^R$
  (which is at least 1 for large enough $n$),
  \[
    \P_2 \left[
      \left|
        Z_v - \frac{{s'}^k}{n-\delta n} M_z M_*
      \right| > \frac{\epsilon {s'}^{k+R}}{n - \delta n} |S_{\ast}|
    \right] \to 0.
  \]
  By Lemma~\ref{lem:Gj_cond}, the same statement holds under $\P_1$,
  conditioned on $J_v = j$.
  \end{proof}

\subsubsection{Branching processes}

The purpose of this section is to show that $M_v$ can be used
to estimate $\sigma_v$. We will do this by exploiting the
connection between neighborhoods in $G$ and multi-type branching processes.
Since this section is the only place where we will use the theory
of branching processes, we will give only a brief introduction;
readers unfamiliar with this theory should consult the book by
Athreya and Ney~\cite{AthreyaNey:72}.

For notational simplicity, we will assume for now that $s > 0$.
The case $s < 0$ will be discussed at the end of the section.
For the rest of this section,
$T$ will denote a Galton-Watson branching process with $\Poisson(d)$
offspring distribution rooted at $\rho$.
We will assign three random labellings to the vertices of $T$ in the following way:
first, divide $T$ into connected components by running
\emph{bond percolation}: deleting each edge independently with probability $s/d$. Then, for each component choose a label uniformly in $\{\pm 1\}$ and assign that label to all vertices in that component.
We define $\eta,\eta^+,$ and $\eta^-$ respectively to be the configurations generated
this way where the connected component of the root is labelled randomly,
labelled $+1$, or labelled $-1$ respectively.
Let $\zeta=\eta_\rho$, let $\Psi_R =\sum_{v\in S_R(\rho)} \eta_v$ and define $\Psi_R^\pm$ similarly.

It is well-known (and not hard to check) that the random labelling
$\eta$ may also be generated in the following way: choose
$\eta_\rho$ uniformly at random. For every child $u$ of $\eta_\rho$
independently, let $\eta_u = \eta_\rho$ with probability $\frac{a}{a+b}$
and otherwise let $\eta_u = -\eta_\rho$. Then recurse this process down
the tree: for every child $w$ of $u$ independently, let $\eta_w = \eta_u$
with probability $\frac{a}{a+b}$ and otherwise let $\eta_w = -\eta_u$.
The processes $\eta^+$ and $\eta^-$ may be generated similarly,
except that instead of beginning with $\eta_\rho$ labelled randomly,
we fix $\eta^+_\rho = +1$ and $\eta^-_\rho = -1$.

\begin{lemma}
  Let $\xi$ be a uniform random variable on $[-1, 1]$ that is
  independent of $T$, $\eta$, and $\eta^\pm$.
There exist $\kappa > 0$ and $\eps > 0$ such that
\begin{equation}\label{e:BetterThenHalf}
\P[\Psi_R^+ \geq \xi \kappa s^R] \geq \frac12 + 2\epsilon.
\end{equation}
\end{lemma}

\begin{proof}
By symmetry, $\Psi_R$ is symmetric about 0 and so if $\xi$ is an independent uniform on $[-1,1]$ then for any $\kappa > 0$,
\[
\P[\Psi_R \geq \xi \kappa s^R] = \frac12.
\]
Moreover, analysis of multi-type branching processes going back to Kesten and Stigum~\cite{KestenStigum:66} shows that $\E \Psi_R^2 = O(s^{2R})$
provided $s^2>d$.  Also, with percolation construction above, it is clear that $\Psi_R^+-\Psi_R^-$ is simply twice the size of the percolation component of $\rho$ intersected with level $R$.  This is exactly given by a branching process with $\hbox{Poisson}(s)$ offspring distribution so
\[
  \liminf_{R\to \infty} \P[\Psi_R^+ - \Psi_R^- \le s^R] \ge \delta
\]
for small enough $\delta > 0$. By Chebyshev's inequality,
both $\Psi_R^+$ and $\Psi_R^-$ belong to $[-\kappa s^R, \kappa s^R]$ with probability
$1 - O(\kappa^{-2})$. Then
\[
\P\left[\frac{\Psi_R^-}{\kappa s^R} \le \xi \le \frac{\Psi_R^+}{\kappa s^R}\right] \geq
\frac 2\kappa \P[\Psi_R^+ - \Psi_R^- \le s^R]
- \P[|\Psi_R^-| \ge \kappa s^R]
- \P[|\Psi_R^-| \ge \kappa s^R]
\ge \frac 2\kappa \delta - O(\kappa^{-2}).
\]
Finally, symmetry of $\Psi_R^+$ and $\Psi_R^-$ implies that
\[
\P[\Psi_R^+ \geq \xi \kappa s^R] \geq
\frac 12 + \P\left[\frac{\Psi_R^-}{\kappa s^R} \le \xi \le \frac{\Psi_R^+}{\kappa s^R}\right] \geq
\frac 12 + \frac 2\kappa \delta - O(\kappa^{-2}),
\]
which completes the proof if $\kappa$ is a sufficiently large constant.
\end{proof}

\begin{lemma}\label{l:treeCoupling}
For $1\leq i \leq \log n$ let $(T_i, \eta_i)$ be iid copies of $(T, \eta)$ above for $R = 2 \lceil \log \log \log \log n \rceil$.  For $v_1,\ldots,v_{\log n}$ be uniformly chosen vertices in $V$,
\[
d_{TV}\left(\{(T_i, \eta_i)\}_{1\leq i \leq \log n},\{(B_R(v_i), \sigma(B_R(v_i))\}_{1\leq i \leq \log n} \right) \to 0
\]
as $n\to\infty$.
\end{lemma}

\begin{proof}
  This argument is a minor variation on a well-known argument showing the
  local tree-like structure of sparse graphs. We will give only a sketch, but
  a much more detailed argument (although for only one neighborhood) is
  given in~\cite{MoNeSl:13}.

We establish the result by coupling the two processes.   By Markov's inequality, with high probability $\sum_{i=1}^{\log n} |T_i| \leq \log^2 n$
and $\sum_{i=1}^{\log n} |B_R(v_i)| \leq \log^2 n$. Moreover,
by standard arguments in sparse random graphs, $\bigcup B_R(v_i)$
is a disjoint union of trees with high probability.

We reveal the branching process trees by sequentially revealing for each vertex how many children of each label it has ($\hbox{Poisson}(a/2)$ of the same label and $\hbox{Poisson}(b/2)$ of the opposite label) down to level $R$ in a breadth-first manner.

Similarly, we can reveal the neighborhoods of the $v_i$ and their labels in $G$ by sequentially revealing the neighbors and labels of the currently revealed vertices.
Suppose that we condition on the labels of all vertices and on the graph
structure that was revealed so far, and suppose that we want to reveal
the neighbors of a given vertex $u$. With high probability, none
of these revealed neighbors will belong to the already-explored set
and so we will focus on $u$'s neighbors among the unexplored vertices.
If $n^\pm$ are the
numbers of $\pm1$-labelled vertices that have not yet been explored,
then $u$ has $\Binom(n^{\sigma_u}, a/n)$ neighbors of label $\sigma_u$
and $\Binom(n^{\sigma_u}, b/n)$ neighbors of label $-\sigma_u$. Note
that $n^{\pm}$ are both in $n/2 \pm n^{2/3}$ with high probability, because
the original labels were biased by at most $O(n^{1/2})$ and we have
revealed at most $\log^2 n$ of them.

We couple these two processes with the usual coupling of Poisson and Binomial random variables. In each step we fail with probability $O(n^{-1/3})$ and (since there are at most $\log^2 n$ steps) the coupling altogether fails with probability $o(1)$.
\end{proof}

Consider the estimator
\[
  \mathcal{A}_{j,v} = \sgn( M_v  + \kappa s^R \xi_{j,v}).
\]
Using the coupling between graphs and trees, we will show that
$\calA_{j,v}$ is a good estimator for $\sigma_v$. Later,
we will show
that $\calA_{j,v}$ usually agrees with our previous estimator 
$\tau_{j,v}$.

\begin{lemma}\label{l:compAandAtilde}
We have that
\[
\P\left[\frac1{|V'|}\sum_{v\in V'} \sigma_v \mathcal{A}_{j,v} \geq (3/2)\epsilon\right]\to 1
\]
\end{lemma}
\begin{proof}
Let $v_1,\ldots,v_{\log n}$ be a uniform sample without replacement from $V'$. Take the coupling in Lemma~\ref{l:treeCoupling},
and let $\Psi_{R,i} = \sum_{v \in S_R(\rho_i)} \eta_v$, where
$\rho_i$ is the root of $T_i$. Set
\[
  \mathcal{A}'_{j,v_i} = \sgn(  \Psi_{R,i}  + \kappa s^R  \xi_{j,v_i}).
\]
\begin{align*}
\P[\sum_{i=1}^{\log n} \sigma_{v_i} \mathcal{A}_{j,v_i} \geq 1.9 \epsilon\log n]
=\P[\sum_{i=1}^{\log n} \sigma_{v_i} \mathcal{A}'_{j,v_i} \geq 1.9 \epsilon\log n] +o(1)
\end{align*}
By~\eqref{e:BetterThenHalf},
$\P[\sigma_{v_i} \sgn(  \Psi_{R,i}  + \kappa s^R  \xi_{i})] \geq 2\epsilon$.
By Hoeffding's inequality, since the $\mathcal{A}'_{j,v_i} $ are independent,
\[
\P[\sum_{i=1}^{\log n} \sigma_{v_i} \mathcal{A}'_{j,v_i}  \geq 1.9 \epsilon\log n] \to 1.
\]
By Lemma~\ref{l:treeCoupling},
the same holds for $\mathcal{A}_{j,v_i}$. If we now partition $V'$ to sets of size $\log n$ and use the fact that $\sigma_{v_i} \mathcal{A}'_{j,v_i}$ are $\pm 1$, we obtain the claim of the lemma.
\end{proof}

Recall that we have been assuming $s > 0$. In the case $s < 0$,
Lemma~\ref{l:compAandAtilde} (which is the only result
from this section that we will use later) remains true. Indeed,
in order to generate the $T$ and $\eta$ for the case $s < 0$,
one can generate them for $|s|$ and then flip the sign of every
label in an odd generation. Since $R$ is even, level $R$
of the tree is unchanged and $s^R = |s|^R$. Thus,
Lemma~\ref{l:compAandAtilde} remains true.

\subsubsection{Accuracy of the estimator $\tau_{J_v,v}$}
Recall that $\tau_{j,v}$ was defined as
\[
  \tau_{j,v} = \sgn\left(\sum_{u \in S_\ast, u' \in S_v}  N^{(k,j)}_{u,u'} + \kappa \frac{{s'}^{k+R+1}}{d' n} |S_{\ast}|  \xi_{j,v}\right),
\]
and that $J_v$ is the first $j$ such that
$B_{R-1}(v) \cap V_j = \emptyset$ and $(S_v \cup S_\ast) \subset V_j$,
and $J_v = 0$ if no such $j$ exists.

\begin{lemma}\label{l:JvBound}
We have that for a random $v\in V'$, $\P[J_v=0]\to 0$.
\end{lemma}
\begin{proof}
Recall that with high probability we have that $|S_\ast|=\sqrt{\log\log n}$.  With high probability $|B_R(v)|\leq d^{2R}$.  Condition on $|B_R(v)|\leq d^{2R}$. The probability that $B_{R-1}\cap V_j=\emptyset$ and $S_v \cup S_\ast \subset V_j$ is bounded below by $e^{-c (d^{2R} +|S_\ast|)} \geq (\log n)^{-1/2}$.  Since these are independent events given $|B_R(v)|$ it follows that with probability tending to one $J_v \neq 0$.
\end{proof}

We now show that the indicators $\mathcal{A}_{J_v,v}$ and $\tau_{J_v,v}$ usually agree.
\begin{lemma}\label{l:compAandB}
\begin{equation} \label{eq:lAB}
  \E \left[\frac1{|V'|}\sum_{v\in V'}  \mathcal{A}_{J_v,v} \tau_{J_v,v} 1_{\{J_v \neq 0\}} \right] \to 1.
\end{equation}
\end{lemma}
\begin{proof}
By Lemma~\ref{l:JvBound}, the probability of $J_v = 0$ goes to zero,
and therefore
Lemma~\ref{l:MainN} implies that
 \begin{equation}\label{e:NBoundA}
   \P\left[\{J_v = 0 \} \cup \left\{ J_v \ne 0 \text{ and } \left|\sum_{u \in S_\ast, u' \in S_v}  N^{(k,J_v)}_{u,u'} - \frac{{s'}^{k} }{n}M_v M_\ast \right| > \frac{\epsilon s^{'k +R}}{n} |S_\ast|
 \right\} \right] \to 0
\end{equation}

The event $\mathcal{A}_{J_v,v} \neq \tau_{J_v,v}$ is equivalent to $\xi_{J_v,v}$ falling outside the interval with end-points
$-M_v/(\kappa s'^R)$ and
\[
-\frac{d' n\sum_{u \in S_\ast, u' \in S_v}  N^{(k,J_v)}_{u,u'}}{|S_\ast| \kappa s'^{k+R+1}}.
\]

Since with high probability $M_\ast$ is concentrated around $\frac{s'}{d} |S_\ast|$, Lemma~\ref{l:MainN} implies the latter end point converges in probability to
\[
-\frac{ d' n \frac{s'^k}{n} \frac{s'}{d'} |S_\ast| M_v}{\kappa \frac{s^{'k+R}}{n} \frac{s'}{d'} |S_\ast|} = -\frac{M_v}{\kappa s'^R}.
\]
Therefore the probability that $-\xi_{J_v,v}$ falls in the interval converges to $0$ as needed.
\end{proof}

We can now complete the proof of Theorem~\ref{thm:alg}.

\begin{proof}[Proof of Theorem~\ref{thm:alg}]
\noindent Combining Lemmas~\ref{l:compAandAtilde}, \ref{l:compAandB} and~\ref{l:JvBound} we have that with high probability
\[
\sum_{v\in V} \tau(v) \sigma(v) \geq \epsilon n.
\]
yielding an algorithm recovering the a constant correlation with the true partition.  The running time bound was proved in
Section~\ref{s:NCalculation}.
\end{proof}

The proof of Theorem~\ref{thm:main-accurate} (i.e., when we are far
above the threshold) is rather easier, and doesn't require
the branching process tools:
\begin{proof}[Proof of Theorem~\ref{thm:main-accurate}]
  We consider a simplified version of the main algorithm, with
$R = 0$ and $\kappa = 0$ (so that $S_v = \{v\}$). (In fact, the algorithm
as stated also works, but the analysis is more tedious, since it requires
reproving Lemmas~\ref{l:treeCoupling} and~\ref{l:compAandAtilde}
with accuracy going to one.)
Theorem~\ref{thm:path-weights-main} implies that
$N_{u,v}^{(k,j)} = (1 + o(1)) \sigma_u \sigma_v s^k/n$ with probability
tending to 1. Together with Lemma~\ref{l:JvBound},
this implies that $\P[\sigma_v = \tau_v] \to 1$.
\end{proof}

\section{Combinatorial path bounds}
A crucial ingredient in the proof is obtaining bounds on the number of various types of paths (in the complete graph) in terms of how much they self-intersect, either  by intersecting a previous vertex on the path or by repeating an edge of the path.

\begin{definition}\label{def:new-old-returning}
Given a path $\gamma = (v_1, \dots, v_k)$, we say that an edge $(v_i, v_{i+1})$
\begin{itemize}
\item
is \emph{new} if for all $j \le i$, $v_j \ne v_{i+1}$
\item
is \emph{old} if there is some $j < i$ such that $\{v_i, v_{i+1}\} = \{v_j, v_{j+1}\}$.
\item
  Otherwise, we say that $(v_i, v_{i+1})$ is \emph{returning} (in this case $v_{i+1} = v_j$ for $j < i$ but $\{v_{i},v_{i+1}\}$ is not one
of the previous edges).
\end{itemize}
Let $k_n(\gamma), k_o(\gamma)$ and $k_r(\gamma)$ be the number of
new, old, and returning edges respectively.
\end{definition}

\begin{definition}
We say that a path $\gamma$ is \emph{$\ell$-tangle-free} if
the graph $(V(\gamma), E(\gamma))$ is $\ell$-tangle-free.
\end{definition}

For the rest of this subsection we fix $\alpha$
and set $k = \lceil \alpha \log n \rceil$.
Note that for every new edge in a path, the number of distinct vertices in the path increases by one,
as does the number of distinct edges. For a returning edge, only the number of edges increases,
while for an old edge, neither increases. Therefore we easily see that:

\begin{claim} \label{claim:edge_vertex_count}
The number of vertices visited by the path $\gamma$ is
$k_n(\gamma) + 1$ and the number of edges is $k_n(\gamma) + k_r(\gamma)$.
\end{claim}

Our first bound is a fairly crude one that will allow us to assume
that $k_r$ is smaller than some constant.
Note that there is no non-backtracking restriction yet.

\begin{lemma}\label{lem:constant-returns}
  For any constant $C$, if $k_r \ge 1$ and $n$ is sufficiently large then
  there are at most
 \[
   n^{k_n + k_r/2 + C \log(2e k_r)}
 \]
 paths $\gamma$ of length at most $C \log n$, with a
 fixed starting and ending point, and satisfying $k_n(\gamma) = k_n$
 and $k_r(\gamma) = k_r$.
\end{lemma}

The point of Lemma~\ref{lem:constant-returns} is that it implies that paths with large $k_r$
are so rare that they do not contribute any weight. Indeed, for some
$\alpha$ to be determined choose $k^*$ large enough
(depending on $\alpha$) so that
\begin{equation} \label{eq:k*}
 4 \alpha \log(2d) + 4 \alpha \log(2 e k^*) - k^*/2 < -4.
\end{equation}
It then follows from Lemma~\ref{lem:constant-returns}
that if $\Gamma$ is the collection of all paths of length at most $4\alpha \log n$
with $k_r(\gamma) \ge k^*$
then
\begin{align*}
 \sum_{\gamma \in \Gamma} \left(\frac{2d}{n}\right)^{k_n(\gamma) + k_r(\gamma)}
 &\le \sum_{k_r \ge k^*} (2d)^{4\alpha \log n} n^{4 \alpha \log (2ek_r) - k_r/2} \\
 &= \sum_{k_r \ge k^*} n^{4 \alpha \log (2d) + 4 \alpha \log (2ek_r) - k_r/2} \\
 &= n^{-4 + o(1)}.
\end{align*}

Note that for any path $\gamma$ and any labelling $\sigma$,
\[
 |\E [X_\gamma \mid \sigma]| \le \left(\frac{2d}{n}\right)^{k_r(\gamma) + k_n(\gamma)}
\]
(since $k_r(\gamma) + k_n(\gamma)$ is the number of edges in $\gamma$). 
Hence, we have:

\begin{corollary}\label{cor:X-many-returns}
Let $k^* = k^*(\alpha, d)$ be defined in~(\ref{eq:k*}). Then
\[
 \sum_{\gamma} | \E [X_\gamma \mid \sigma] |
 \le n^{-4 + o(1)},
\]
where the sum ranges over $\gamma$ of length at most $4\alpha \log n$
and with $k_r(\gamma) \ge k^*$.
\end{corollary}

\begin{proof}[Proof of Lemma~\ref{lem:constant-returns}]
  Consider paths of fixed length $k$; later, we will sum over all
  $k \le C \log n$.
Suppose that for all $i$, we decide in advance whether $(v_i, v_{i+1})$ will be new, old, or returning.
There are at most $\binom{k}{k_n\ k_o\ k_r}$ ways to make this choice.
Fix an $i$ and suppose that $v_i$ has already been determined.
If $(v_i, v_{i+1})$ is new then there are at most $n$ choices for $v_{i+1}$.
If $(v_i, v_{i+1})$ is returning then there are at most $|V(\gamma)| = k_n + 1 \le k$ choices for $v_{i+1}$.
Otherwise, $(v_i,v_{i+1})$ is an old edge, and there are at most
$k_r + 2$ choices for $v_{i+1}$ because $k_r+2$ bounds the maximum degree
of the final path.
Hence, the total number of choices is at most
\begin{align*}
 \binom{k}{k_n\ k_o\ k_r} n^{k_n} k^{k_r} (k_r + 2)^{k_o}
 &\le \frac{k^{k_o + k_r}}{k_o! k_r!} n^{k_n} k^{k_r} (2k_r)^{k_o} \\
 &= n^{k_n} \frac{k^{2 k_r}}{k_r!} \frac{(2 k k_r)^{k_o}}{k_o!} \\
 &\le n^{k_n} \left(\frac{e k^2}{k_r}\right)^{k_r}
 \left(\frac{2e k k_r}{k_o}\right)^{k_o} \\
 &= n^{k_n + k_r} \left(\frac{e k^2}{n k_r}\right)^{k_r}
 \left(\frac{2e k k_r}{k_o}\right)^{k_o}.
\end{align*}
(In the case $k_o = 0$, we adopt the convention $(y/0)^0 = 1$.)
Now, the quantity $(y/x)^x$ is increasing in $x$ as long as $x \le y/e$.
Applying this with $y = 2e k k_r$ and the values $x = k_o \leq k \leq y/e$
 we  have
\[
 \left(\frac{2e k k_r}{k_o}\right)^{k_o}
 \le (2e k_r)^k \le (2ek_r)^{C \log n} = n^{C \log(2e k_r)}.
 \]
 On the other hand, $e k^2 / (n k_r) \le n^{-2/3}$
 for sufficiently large $n$.
 Hence, the total number of paths of length $k$ is at most
 \[
   n^{k_n + k_r} n^{-2k_r/3} n^{C \log (2e k_r)}.
 \]
 Summing over $k \le C \log n$ introduces an extra factor of
 $C \log n$, but this factor is cancelled out by $n^{-k_r/6}$
 for sufficiently large $n$.
\end{proof}

The bounds of Lemma~\ref{lem:constant-returns} are not accurate when
$k_r$ is small. Essentially, we require bounds of $n^{k_n-1 + o(1)}$
in order to make the rest of our argument work (certainly, we can't expect
any better bounds, since every new edge but the last one has almost
$n$ choices). In order to achieve
this bound, we need to introduce extra structure into our paths:
they need to be non-backtracking and without many tangles.

\begin{definition}
  Consider the path $\gamma$ as a multigraph (i.e.\ each edge has a
  multiplicity according to the number of times $\gamma$ crosses it).
  We say that a path has $t$ $\ell$-tangles if $t$ is the minimal
  number of edges (counting multiplicity) that need to be deleted from $\gamma$
  in order to make it $\ell$-tangle-free.
\end{definition}

\begin{lemma}
  \label{lem:few-tangles}
If $k_r \ge 1$ then there are at most
\[
 k^{5 k_r + 4 k_r k/\ell + 8 k_r t} n^{k_n-1}
\]
paths with $t$ $\ell$-tangles that have a fixed starting and ending point,
and that satisfy $k_n(\gamma) = k_n$ and $k_r(\gamma) = k_r$.
\end{lemma}

In order to see the use of Lemma~\ref{lem:few-tangles}, note
that if $k = O(\log n)$, $k_r = O(1)$, and $\ell = \omega(\log \log n)$ then
the bound in Lemma~\ref{lem:few-tangles} is of the order
$k^{O(t)} n^{k_n - 1 + o(1)}$. We will argue later that having many
tangles results in a small path weight with high probability,
and so our bound is effectively of the order $n^{k_n - 1 + o(1)}$.

\begin{proof}
First, note that if we
specify which edges are returning and we also specify the first new
edge after each returning edge, then we have also determined
which edges are old
(because every edge after a new edge but before the next returning edge is
new). Therefore, the number of ways to specify
which edges are old, new, or returning is at most $k^{2k_r}$.
Then there are at most $k^{k_r}$ ways to choose the
returning edges and at most $n^{k_n-1}$ ways to choose the
new edges (since one of them must hit the final vertex, so it has
no choices). So far, we have made at most
$k^{3k_r} n^{k_n-1}$ choices, and these choices
determine the edges traversed by $\gamma$.

Having fixed the edges traversed by $\gamma$,
we will now count old edges.
We denote by $d(v) = d(v, \gamma)$ the degree of $v$ in $\gamma$:
that is, the number of $w \in \gamma$
such that $\{w, v\} \in E(\gamma)$. Let $\Vints$ be the set of vertices
with degree at least 3. Note that because $\gamma$ is non-backtracking,
if an edge $(v_i, v_{i+1})$ is old, then $v_{i+1}$ is already
determined by the path up to $v_i$ unless $v_i \in \Vints$.

Let $\calT$ be the collection of minimal sets $T \subset E(\gamma)$
such that $E(\gamma) \setminus T$ has no $\ell$-tangles. Given $T \in \calT$,
we say that a neighbor $w$ of $v \in \Vints$ is
\emph{short} if there is a cycle in $E(\gamma) \setminus T$ that contains the
edge $\{v, w\}$
and has length at most $2\ell$; we say that $w$ is \emph{tangled}
if it is not short, but there is a cycle in $\gamma$ that
contains the edge $\{v, w\}$ and has length at most $2\ell$;
otherwise we say that $w$ is \emph{long}.
Note that the set of long neighbors is independent of $T$, and
that every $v \in \Vints$ has at most two short neighbors. Now we change the
order: for every $v \in \Vints$ choose up to two of its neighbors to be
short, and say that $T \in \calT$ is compatible with this choice if the
chosen short and tangled neighbors agree with the definition above.
Note that there are at most $\prod_{d \in \Vints} d(v)^2$ ways to choose the
collection of all short neighbors for all $v \in \Vints$.
For a choice of short neighbors, we say that $\gamma$ has $t$
\emph{$\ell$-tangles
with respect to this choice} if there exists some $T$ compatible
with the choice of short neighbors that is crossed at most $t$ times.
Note that if $\gamma$ has $t$ $\ell$-tangles then there is some
choice of short neighbors such that $\gamma$ has $t$ $\ell$-tangles
with respect to that choice.

Now fix a choice of short edges for every $v \in \Vints$; we will bound
the number of $\gamma$ that have $t$ $\ell$-tangles with respect to
this choice.
For a given $v \in \Vints$, let $m(v)$ be the number of times that $v$ was visited.
Let $\longIn(v)$ be the number of times that $v$
was visited from a long or tangled neighbor and
$\longOut(v)$ be the number
of the times we move to a long or tangled neighbor.
Note that when $\gamma$ arrives at $v$ on an old edge from a short vertex and leaves on an old edge to another short vertex, then the edge leaving $v$ is determined by the fact that $v$ has exactly two short neighbors.
At all other times that $\gamma$ leaves $v$ on an old edge, there are at most $d(v)$ choices for the outgoing edge.
Thus the total number of ways to chose old edges starting at $v$ is bounded by
\[
m(v)^{\longOut(v)+\longIn(v)} d(v)^{\longIn(v) + \longOut(v)} \leq m(v)^{2\longOut(v)+2\longIn(v)},
\]
where $m(v)^{\longOut(v) + \longIn(v)}$ bounds the number of ways that we
can intersperse the short arrivals and departures among all visits to $v$,
and the second inequality follows from the fact that $d(v) \leq m(v)$.
Repeating this for
all $v \in \Vints$, we see that the number of ways to choose all the old edges
in $d$ is at most
\[
 \prod_{v \in \Vints}
 m(v)^{2 \longOut(v)+2 \longIn(v)}.
\]

Next, we use the tangle structure to bound $\longIn(v)$ and
$\longOut(v)$. Indeed, after leaving $v$ via a long neighbor, we must wait
at least $2\ell$ steps before visiting $v$ again; after leaving $v$
via a tangled neighbor, we must either wait at least $2\ell$ steps
before returning or else for every $T \in \calT$ that is compatible
with the choice of short neighbors, we must pass through an edge of $T$.
Hence, $\longOut(v) \le k / (2\ell) + t$.
A similar argument shows that $\longIn(v) \le k/(2\ell) + t$,
and hence the number of ways to choose the old edges is at most
\[
  \left(\prod_{v \in \Vints} m(v)\right)^{2k/\ell + 4t}.
\]

To put everything together, there were at most
$k^{3k_r} n^{k_n - 1}$ ways to fix the edge types and the edge
set of $\gamma$. Then there were at most $\prod_{v \in \Vints} d(v)^2 \le
\prod_{v \in \Vints} m(v)^2$ ways to choose the short neighbors. For
each such choice, there were at most 
$\big(\prod_{v \in \Vints} m(v)\big)^{2k/\ell + 4t}$ ways to choose
old edges such that the resulting path would have $t$ $\ell$-tangles
with respect to the choice of short neighbors. All together, this
gives at most
\begin{equation}\label{eq:few-tangles-1}
  k^{3 k_r} n^{k_n - 1} \left(\prod_{v \in \Vints} m(v)\right)^{2k/\ell + 2 + 4t}
\end{equation}
paths.

By the AM-GM inequality,
\[
 \prod_{v \in \Vints} m(v)
 \le \left(\frac{1}{|\Vints|} \sum_{v \in \Vints} m(v)\right)^{|\Vints|}
 \le \left(\frac{1}{|\Vints|} \sum_{v \in \Vints} m(v)\right)^{2k_r},
\]
where the second inequality follows because every time the walk returns to
its old path, it creates at most two vertices of degree higher than two
(one when
the walk returns, and one when it leaves again). Since $m(v) \le k$ for
every $v$, the quantity above is bounded by $k^{2k_r}$.
Plugging this back into~\eqref{eq:few-tangles-1}, we get the claimed bound.
\end{proof}

\subsection{Pairs of self-avoiding paths}

When we take second moments over various sums over paths, we will
end up having to control the number of pairs of paths with certain
properties. In what follows, we take two self-avoiding
paths, $\gamma_1$ and $\gamma_2$, of length $k$. We will refine
Definition~\ref{def:new-old-returning} by saying that
a (directed) edge $(u, v)$ of $\gamma_2$ is
\emph{new with respect to $\gamma_1$} if $v \not \in V(\gamma_1)$.
We say that $(u, v)$ is \emph{old with respect to $\gamma_1$}
if the (undirected) edge $\{u, v\}$
appears in $\gamma_1$. Otherwise, we say that
that $(u, v)$ is \emph{returning with respect to $\gamma_1$}.
We write $k_{n,\gamma_1}(\gamma_2)$, $k_{o,\gamma_1}(\gamma_2)$,
and $k_{r,\gamma_1}(\gamma_2)$ for the numbers of edges of these
three types in $\gamma_2$.

\begin{lemma}\label{lem:two-saw-paths}
  Fix vertices $u, u', v, v'$ (not necessarily distinct). There are at
  most
  \[
    2 (k+1) \binom{k}{k_{r,\gamma_1}} \binom{k}{k_{r,\gamma_1} + 1} (2k)^{k_{r,\gamma_1}} n^{k + k_{n,\gamma_1} - 1 - 1_{v' \not \in \{u, v\}}}
  \]
  pairs $(\gamma_1, \gamma_2)$ of length-$k$ self-avoiding paths where
  $\gamma_1$
  goes from $u$ to $v$, $\gamma_2$ goes from $u'$ to $v'$, and where
  $k_{n,\gamma_1}(\gamma_2) = k_{n,\gamma_1}$ and
  $k_{r,\gamma_1}(\gamma_2) = k_{r,\gamma_1}$.
\end{lemma}

\begin{proof}
  For this proof, whenever we speak of old, new, or returning edges
  of $\gamma_2$, we mean with respect to $\gamma_1$.

  First, assume that $v'$ is not an interior node of $\gamma_1$.
  There are at most $n^{k-1}$ such choices for $\gamma_1$; fix one
  and consider $\gamma_2$. Every sequence of old edges in $\gamma_2$
  either occurs at the beginning of $\gamma_2$, or it is preceded by
  a returning edge. Hence, there are at most
  $\binom{k}{k_{r,\gamma_1}}\binom{k}{k_{r,\gamma_1} + 1}$ choices
  for the edge types of $\gamma_2$: $\binom{k}{k_{r,\gamma_1}}$ choices for which
  edges are returning, and at most $\binom{k}{k_{r,\gamma_1}+1}$ choices for
  the end of each sequence of old edges. Each new edge has at most
  $n$ choices for its endpoint. In the case that $v' \not \in \{u, v\}$
  then (since it is also not an interior node of $\gamma_1$) the
  last edge is new, but it has no choices.
  Hence, there are at most $n^{k_{n,\gamma_1} - 1_{v' \not \in \{u, v\}}}$
  choices for the new edges.
  Each returning edge has at most
  $|E(\gamma_1)| = k$ choices for its endpoint, and every sequence
  of old edges has at most $2$ choices: it must follow the
  (self-avoiding) path
  $\gamma_1$, but it may do so in either direction; moreover,
  there are at most $k_{r,\gamma_1}+1$ distinct sequences of old edges. Hence, the
  total number of choices for $\gamma_2$ is bounded by
  \[
    \binom{k}{k_{r,\gamma_1}} \binom{k}{k_{r,\gamma_1} + 1} 2 ^{k_{r,\gamma_1} + 1}k^{k_{r,\gamma_1}} n^{k_{n,\gamma_1}}.
  \]
  Therefore there are at most
  \[
    2 \binom{k}{k_{r,\gamma_1}} \binom{k}{k_{r,\gamma_1} + 1} (2k)^{k_{r,\gamma_1}} n^{k + k_{n,\gamma_1} - 1 - 1_{v' \not \in \{u, v\}}}
  \]
  pairs that satisfy the conditions of the lemma, and also the
  additional constraint that $v'$ is not an interior node of $\gamma_1$.

  Now suppose that $v'$ is an interior node of $\gamma_1$. There are
  at most $k n^{k-2}$ ways to choose such a $\gamma_1$. We may
  repeat the previous paragraph to bound the number of choices
  of $\gamma_2$, except that this time there will be up to
  $n^{k_{n,\gamma_1}}$ choices for the new edges, because the final
  edge of $\gamma_2$ may not be new. Hence, there are at most
  \[
   2 k \binom{k}{k_{r,\gamma_1}} \binom{k}{k_{r,\gamma_1} + 1} (2k)^{k_{r,\gamma_1}} n^{k + k_{n,\gamma_1} - 2}
  \]
  pairs of paths of this form, where $v'$ is an interior node of $\gamma_1$.
  Combined with the other case, this proves the claim.
\end{proof}

\section{Weighted sums over self-avoiding paths}\label{sec:self-avoiding}

In this section, we consider the behavior of weighted sums over
self-avoiding paths.
In particular, we will prove~\eqref{eq:thm-first-moment},~\eqref{eq:thm-second-moment}, and~\eqref{eq:thm-cross-moments} from Theorem~\ref{thm:path-weights-main}.

\subsection{The weight of self-avoiding walks and simple cycles}

Eventually, we will need to bound (or bound) the expected weight
of complicated paths. Our basic building block for these computations
is the expected weight of a self-avoiding path.

\begin{lemma}\label{lem:path}
  Let $\zeta$ be either a self-avoiding path or a simple cycle.
  Let $z$ be the length of $\zeta$ and let $u, v$ be the endpoints.
  If $pmz = n^{o(1)}$ then uniformly with respect to $\gamma$
  \[
    \E \left[\prod_{e \in \zeta} W_e^m \mid \sigma_u, \sigma_v \right]
    = \begin{cases}
      \frac{\sigma_u \sigma_v s^z}{n^z} &\text{if $m = 1$} \\
      (1 + n^{-1 + o(1)})\frac{\sigma_u \sigma_v s^z + d^z}{n^z} &\text{if $m \ge 2$} \\
    \end{cases}
  \]
\end{lemma}

\begin{proof}
  First, consider the case $m = 1$.
  For any labelling $\tau$ that is compatible with $\sigma_u$ and $\sigma_v$,
  \[
    \E \left[\prod_{e \in \zeta} W_e \mid \tau \right]
    = \prod_{(x,y) \in \zeta} \frac{\tau_{x} \tau_{y} s}{n}.
  \]
  Since $\zeta$ is a path, if $x$
  is an interior vertex of $\zeta$ then $\tau_x$ appears exactly
  twice in the product above.
  Since $\tau_x^2 = 1$, these terms all cancel out,
  leaving
  \[
    \E \left[\prod_{e \in \zeta} W_e \mid \tau \right]
    = \frac{\tau_{u} \tau_{v} s^z}{n^z},
  \]
  which proves the claim in the case $m=1$.

  For the case $m\ge2$, note that
  \[
    \E [W_{(x,y)}^m \mid \tau]
    = (-d/n)^m \left(1 - \frac{\tau_x \tau_y s + d}{n}\right)
    + (1 - d/n)^m \frac{\tau_x \tau_y s + d}{n}
    = (1 + O(dn^{-1}))^m \frac{\tau_x \tau_y s + d}{n}.
  \]
  Hence,
  \begin{equation}\label{eq:m-ge-2}
    \E \left[\prod_{e \in \zeta} W_e^m \mid \tau \right]
    = (1 + O(dn^{-1}))^{mz} \prod_{(x,y) \in \zeta} \frac{\tau_{x} \tau_{y} s + d}{n}.
  \end{equation}
  Now we take the average over all assignments $\tau$ that agree with $\sigma_u$ and $\sigma_v$:
  \[
   2^{-(z - 1)} \sum_\tau \prod_{(x, y) \in \zeta} \frac{\tau_x \tau_y s + d}{n}
   = \frac{\sigma_{u} \sigma_{v} s^{z} + d^{z}}{n},
  \]
  where the sum ranges over all $2^{z - 1}$ labellings $\tau$ on $\zeta$ that agree with $\sigma_u$ and $\sigma_v$. Combining this with~\eqref{eq:m-ge-2} completes the proof.
\end{proof}

\subsection{Decomposition into segments}

Although our current goal is to understand the contribution of
self-avoiding paths, in order to compute the second moment
in Theorem~\ref{thm:path-weights-main}, we will need to consider the
concatenation of two self-avoiding paths (which may not be self-avoiding).
Therefore, we introduce the following method for decomposing a general
path into its self-avoiding pieces.
This decomposition will also be useful in Section~\ref{sec:tangled},
where we consider more complicated paths.

\begin{definition}\label{def:SAW}
 Consider a path $\gamma$. We say that a collection of paths $\zeta^{(1)}, \dots, \zeta^{(r)}$
 is a SAW-decomposition of $\gamma$ if
 \begin{itemize}
  \item each $\zeta^{(i)}$ is a self-avoiding path;
  \item the interior vertices of each $\zeta^{(i)}$ are not contained in any other
  $\zeta^{(j)}$, nor is any interior vertex of $\zeta^{(i)}$ equal to the starting
  or ending vertex of $\gamma$; and
  \item the $\zeta^{(i)}$ cover $\gamma$, in the sense that
    $E(\gamma) = \bigcup_i E(\zeta^{(i)})$.
 \end{itemize}

  Given a SAW decomposition as above, we let $\Vends$ denote the set of vertices that are the endpoint
  of some $\zeta^{(i)}$ and we let $m_i$ denote the number of times that $\zeta^{(i)}$ was
  traversed in $\gamma$.
\end{definition}

Note that since $\zeta^{(i)}$ and $\zeta^{(j)}$ share no interior vertices,
every time that the path $\gamma$ begins to traverse $\zeta^{(i)}$, it must finish
traversing $\zeta^{(i)}$. Moreover, the fact that $\zeta^{(i)}$ and $\zeta^{(j)}$
share no interior vertices implies that they are edge-disjoint, and so for each
fixed $i$, every edge in $\zeta^{(i)}$ is traversed the same number (i.e. $m_i$) of times.

There is a natural way to construct a SAW-decomposition of a path $\gamma$.
Consider a path $\gamma$ between $u$ and $v$, and let $\Vints$ be the subset of
$\gamma$'s vertices that have degree 3 or more in $\gamma$. Let
\begin{equation}\label{eq:D}
\Vends = \Vints \cup \{u,v\} \cup\{w \in \gamma: \gamma
\text{ backtracks at } w\}.
\end{equation}
Then
$\gamma$ may be decomposed into a collection of self-avoiding walks between vertices
in $\Vends$. To be precise, suppose that $\gamma$ is given by
$u = u_0, u_1, \dots, u_k = v$. Let $j_1 > 0$ be minimal so that $u_{j_1} \in \Vends$ and let
$\gamma^{(1)}$ be the path $u_0, \dots, u_{j_1}$. Inductively, if $j_{i-1} < k$ then
let $j_i > j_{i-1}$ be minimal
so that $u_{j_i} \in \Vends$ and set $\gamma^{(i)}$ to be the path $u_{j_{i-1}}, \dots, u_{j_i}$.
It follows from this definition that the interior nodes in each $\gamma^{(i)}$ are degree
2 in $\gamma$; hence, each $\gamma^{(i)}$ is self-avoiding, and any pair $\gamma^{(i)},
\gamma^{(j)}$ are either identical, or their interior vertices are disjoint.
Finally, let $\{\zeta^{(1)}, \dots, \zeta^{(r)}\}$ be $\{\gamma^{(i)}\}$, but with
duplicates removed.

\begin{definition}\label{def:canonical-SAW}
 We call the preceding construction of $\zeta^{(1)}, \dots, \zeta^{(r)}$ the
 \emph{canonical SAW-decomposition of $\gamma$}.
\end{definition}

We remark that the same construction works for any set $\Vends$ that is larger than the
one defined in~\eqref{eq:D}.

\begin{definition}\label{def:U-canonical-SAW}
 For a set of vertices $U$, if we run the preceding construction, but with
 \[
\Vends = U \cup \Vints \cup \{u,v\} \cup\{w \in \gamma: \gamma
\text{ backtracks at } w\}.
 \]
 instead of as defined in~\eqref{eq:D}, then we call the resulting SAW-decomposition
 the \emph{$U$-canonical SAW-decomposition of $\gamma$.}
\end{definition}

\begin{lemma}\label{lem:r-bound}
 If $\zeta^{(1)}, \dots, \zeta^{(r)}$ is the canonical SAW-decomposition of $\gamma$
 then $r \le 2 k_r(\gamma) + B(\gamma) + 1$, where $B(\gamma)$ is the number of backtracks in $\gamma$.

 If $\zeta^{(1)}, \dots, \zeta^{(r)}$ is the $U$-canonical SAW-decomposition of $\gamma$
 then $r \le 2 k_r(\gamma) + B(\gamma) + 1 + |U|$.
\end{lemma}

\begin{proof}
 Every returning edge in $\gamma$ increases $r$ by at most 2, since it can create a new
 SAW component, and it can split an existing component into 2 pieces.
 Every backtrack in $\gamma$ increases $r$ by at most 1, since it can create a new
 SAW component. This proves the first statement; to prove the second, note that each vertex $v \in A$
 creates at most one new component, since if $v \in \Vints$ then it has no effect, while if
 $v \not \in \Vints$ then it has degree at most 2 in $\gamma$ and so splitting the path that goes
 through $v$ introduces at most one new component.
\end{proof}

\subsection{The weight of a SAW-decomposition}

We can compute the expected weight of a SAW-decomposition by
simply applying Lemma~\ref{lem:path} to each component.
We state the following lemma slightly more generally, so that we
may also apply it to subsets of the SAW-decomposition of a path.

\begin{lemma}\label{lem:non-tangled-segments}
 Suppose $\zeta^{(1)}, \dots, \zeta^{(r)}$ are self-avoiding paths, where
 $\zeta^{(i)}$ is a path of length $z_i$ between $u_i$ and $v_i$.
 Suppose also that no $\zeta^{(i)}$ intersects an interior vertex
 of any other $\zeta^{(j)}$.
 Let $m_1, \dots, m_r \le n^{o(1)}$ be positive integers.
 Suppose that $\bigcup_i \zeta^{(i)}$ has at most $n^{o(1)}$ edges.
 Let $\Vends = \{u_1, v_1, \dots, u_r, v_r\}$.
 If $d \sum_i m_i z_i = n^{o(1)}$ then, uniformly over
 $\gamma$ and over all labellings $\sigma_\Vends$ on vertices in $\Vends$,
 \[
  \E \left[\prod_{i} \prod_{e \in \zeta^{(i)}} W_e^{m_i} \mid \sigma_\Vends \right]
  = (1 + n^{-1+o(1)}) \prod_{i: m_i = 1} \frac{\sigma_{u_i} \sigma_{v_i} s^{z_i}}{n^{z_i}}
  \prod_{i: m_i > 1} \frac{\sigma_{u_i} \sigma_{v_i} s^{z_i} + d^{z_i}}{n^{z_i}}.
\]
Moreover,
\[
  \left|\E \left[\prod_{i} \prod_{e \in \zeta^{(i)}} W_e^{m_i} \mid \sigma_\Vends \right]\right|
\le (1 + o(1)) 2^{r} s^k n^{-|E(\gamma)|},
\]
where $k = \sum_i \sum_{e \in \zeta^{(i)}} m_e$.
\end{lemma}

\begin{proof}
  To get the first claim, note that conditioned on
  $\sigma_{\Vends}$, the terms
  $\prod_{e \in \zeta^{(i)}} W_e^{m_i}$
  are independent as $i$ varies and apply Lemma~\ref{lem:path}
  to each $\zeta^{(i)}$.

  For the claimed inequality,
\begin{align*}
  \prod_{i: m_i = 1} \frac{s^{z_i}}{n^{z_i}}
  \prod_{i: m_i > 1} \frac{s^{z_i} + d^{z_i}}{n^{z_i}}
  &\le 2^{r}
  \prod_{i: m_i = 1} \frac{s^{z_i}}{n^{z_i}}
  \prod_{i: m_i > 1} \frac{d^{z_i}}{n^{z_i}} \\
  &= 2^{r} \prod_{\stackrel{e \in \gamma}{m_e = 1}} \frac{s}{n}
 \prod_{\stackrel{e \in \gamma}{m_e > 1}} \frac{d}{n} \\
 &\le 2^{r} n^{-|E(\gamma)|} s^k,
\end{align*}
where the last inequality follows because $1 < d < s^2$,
and so $m_e > 1$ implies $d < s^2 \le s^{m_e}$.
\end{proof}

\subsection{Proof of~\eqref{eq:thm-first-moment}--\eqref{eq:thm-cross-moments}}\label{sec:saw-paths-proofs}

Now we prove the first three parts of Theorem~\ref{thm:path-weights-main}. The claim~\eqref{eq:thm-first-moment} about
the first moment follows from Lemma~\ref{lem:path}.

\subsubsection{The second moment}\label{sec:second-mom}
For the second moment, we will expand the square in $Y_{u,v}^2$.
Suppose $\gamma_1$ and $\gamma_2$ are a pair of self-avoiding paths of
length $k$
from $u$ to $v$.  By reversing $\gamma_2$ and appending
it to $\gamma_1$, we obtain a single path ($\gamma$, say)
from $u$ to itself which passes through $v$ and backtracks at most
once (at $v$).
We consider the set of all $\gamma$ that can be obtained in this way, and divide them into four classes:
\begin{itemize}
  \item $\Gamma_0$ is the collection of such paths with $k_r(\gamma) = 0$.
    These paths begin with a self-avoiding
   walk from $u$ to $v$, after which they backtrack at $v$ and walk back to $u$ along exactly the same path.
 They have $k$ edges, $k-1$ vertices, and every edge is visited twice.

 \item $\Gamma_1$ is the collection of such paths with $k_r = 1$.
   These paths consist of a simple cycle that is traversed once,
   with up to two ``tails'' that are traversed twice each.

 \item $\Gamma_2$ is the collection of such paths with $2 \le k_r \le k^*$.

 \item $\Gamma_3$ is the collection of such paths with $k_r > k^*$.
\end{itemize}

First, we consider $\Gamma_0$. By Lemma~\ref{lem:path}, if
$U$ does not intersect the interior of $\gamma \in \Gamma_0$ then
$|\E [X_\gamma \mid \sigma_U] | = (1 + o(1)) 2 (d/n)^k$. There are
at most $n^{k-1}$ such paths, giving a total weight of at most
$(1 + o(1)) 2 d^k / n$.
On the other hand,
the contribution of $\gamma$ whose interiors do intersect with
$U$ is negligible: by Lemma~\ref{lem:non-tangled-segments} applied
to the $U$-canonical SAW-decomposition of $\gamma$,
$\E [X_\gamma \mid \sigma_U] = (1 + o(1)) 2^r (d/n)^k$, where $r$
is the number of interior vertices of $\gamma$ that intersect $U$.
On the other hand, the number of such paths $\gamma$ is
at most $|U|^r n^{k-1-r}$; since
$\sum_{r \ge 1} 2^r |U|^r n^{-r} = n^{-1 + o(1)}$, we see that these
contribute only a lower-order term. Hence,
\[
  \left|\sum_{\gamma \in \Gamma_0} \E [X_\gamma \mid \sigma_U]
    \right| \le (1 + o(1)) 2 d^k / n\le (1 + o(1)) 2 s^{2k} n^{-3},
\]
where the second inequality follows from our choice of $k$
in Theorem~\ref{thm:path-weights-main}.
In particular, this term is of a lower order than the bound
claimed in the theorem.

Next, we consider $\Gamma_1$. Recall that the first $k$ steps
of $\gamma \in \Gamma_1$ make up a simple path.
Let $i$ be minimal so that the $(k+i+1)$th step of $\gamma$ is new; let $j$ be such that the $2k-j$th
step of $\gamma$ is returning. It follows that the first $j$ edges of $\gamma$ consist of a simple
path where each edge is traversed twice. The same holds for edges $k-i+1$ through $k-1$. The rest
of $\gamma$ consists of a simple cycle of length $2k - 2(i + j)$, each edge of which is traversed once.
Let $\Gamma_1(i, j)$ denote the set of such paths.
By Lemma~\ref{lem:non-tangled-segments}, if $\gamma$'s interior does not intersect $U$ then
the expected weight of $\gamma \in \Gamma_1(i, j)$ is  bounded by
\[
  |\E [X_\gamma \mid \sigma_U]| \le (1 + o(1)) 2 (d/n)^{i+j} (s/n)^{2k - 2(i+j)}.
\]
Now, $|\Gamma_1(i, j)| = (1 + o(1)) n^{2k - i - j - 2}$ because $\gamma \in \Gamma_1(i, j)$
has $2k - i - j$ distinct vertices (including $u$ and $v$), and once those vertices and their order
is fixed then $\gamma$ is determined.
As in the argument for $\Gamma_0$, the paths whose interiors intersect
$U$ provide a negligible contribution, and hence
\[
  \left|\sum_{\gamma \in \Gamma_1(i,j)} \E [X_\gamma \mid \sigma_U]
    \right| = (1 + o(1)) 2 n^{-2} s^{2k - 2(i+j)} d^{i+j}
 = (1 + o(1)) 2 n^{-2} s^{2k} \left(\frac{d}{s^2}\right)^{i+j}.
\]
Summing over $i$ and $j$, we have
\begin{equation}\label{eq:Gamma_1''}
 \left|\sum_{\gamma \in \Gamma_1} \E [X_\gamma \mid \sigma_U] \right|
 \le (1 + o(1)) 2 n^{-2} s^{2k} \sum_{i,j = 0}^\infty \left(\frac{d}{s^2}\right)^{i+j}
 = (1 + o(1)) 2 n^{-2} s^{2k} \left(\frac{s^2}{s^2 - d}\right)^2.
\end{equation}
Hence, $\Gamma_1$ provides the main term in the claimed bound.

Next, we consider $\Gamma_2$, which we will split up
according to the number of new edges: let $\Gamma_{2,k_n}$
be the set of $\gamma \in \Gamma_2$ with $k_n(\gamma) = k_n$.
To estimate the size of $\Gamma_{2,k_n}$, we apply
Lemma~\ref{lem:two-saw-paths} with $u' = u$ and $v' = v$.
Since $k_n(\gamma) = k - 1 + k_{n,\gamma_1}(\gamma_2)$ and
$k_{r,\gamma_1}(\gamma_2) \le k^* = O(1)$,
Lemma~\ref{lem:two-saw-paths} yields $|\Gamma_{2,k_n}| \le n^{k_n - 1 + o(1)}$.
On the other hand, Lemma~\ref{lem:non-tangled-segments} implies 
that for $\gamma \in \Gamma_2$,
$\E [X_\gamma \mid \sigma_U] \le s^{2k} n^{-k_n(\gamma) - k_r(\gamma) + o(1)}$
since $k_n(\gamma) + k_r(\gamma)$ is the number of distinct edges in
$\gamma$, and since $r$ in Lemma~\ref{lem:non-tangled-segments}
is bounded by $4 k_r + 1 \le 4 k^* + 1 = O(1)$.
Recalling that $k_r(\gamma) \ge 2$ for all $\gamma \in \Gamma_2$,
\[
  \sum_{\gamma \in \Gamma_{2,k_n}} \E [X_\gamma \mid \sigma_U]
  \le s^{2k} n^{-k_r(\gamma) - 1 + o(1)} \le s^{2k} n^{-3 + o(1)}.
\]
Summing over the $k$ choices of $k_n$ shows that the paths in $\Gamma_2$
contribute a smaller order term than the paths in $\Gamma_1$.

Finally, we bound $\Gamma_3$ using Corollary~\ref{cor:X-many-returns};
these terms also contribute a smaller order term.

\subsubsection{The cross moment}\label{sec:cross-mom}

To compute the cross moment in Theorem~\ref{thm:path-weights-main},
we expand the product $Y_{u,v} Y_{u',v'}$ and divide pairs
of paths $\gamma_1 \in \SAW_{u,v}$, $\gamma_2 \in \SAW_{u',v'}$
into three groups:
\begin{itemize}
  \item $\Gamma_0$ are the pairs of paths that do not intersect.
  \item $\Gamma_1$ are the pairs of paths that do intersect,
    and that satisfy $k_{r,\gamma_1}(\gamma_2) \le k^*$.
  \item $\Gamma_2$ are the pairs of paths that satisfy
    $k_{r,\gamma_1}(\gamma_2) > k^*$.
\end{itemize}

For $(\gamma_1, \gamma_2) \in \Gamma_0$, the variables
$X_{\gamma_1}$ and $X_{\gamma_2}$ are independent, and hence
\[
  \E [X_{\gamma_1} X_{\gamma_2} \mid \sigma_U, \sigma_{U'}]
  = \E [X_{\gamma_1} \mid \sigma_U, \sigma_{U'}] \E[X_{\gamma_2} \mid \sigma_U, \sigma_{U'}].
\]
We recall from~\eqref{eq:thm-first-moment}
that the right hand side above is of the order $s^{2k} n^{-2}$;
in order to prove the claim about the cross moments, we need to show
that the contributions of $\Gamma_1$ and $\Gamma_2$ are of the order
$s^{2k} n^{-3 + o(1)}$.

To control $\Gamma_1$, we split pairs of paths according
to $k_{n,\gamma_1}(\gamma_2)$. If $\Gamma_{1,k_{n,\gamma_1}}$ is
the set of pairs of paths in $\Gamma_1$ with $k_{n,\gamma_1}(\gamma_2) = k_{n,\gamma_1}$
then Lemma~\ref{lem:two-saw-paths} implies that
$|\Gamma_{1,k_{n,\gamma_1}}| \le n^{k + k_{n,\gamma_1} - 2 + o(1)}$
(when applying Lemma~\ref{lem:two-saw-paths}, recall that $v'$
is distinct from $u$ and $v$). By Lemma~\ref{lem:non-tangled-segments},
and noting that $|E(\gamma_1) \cup E(\gamma_2)| \ge k + k_{n,\gamma_1} + 1$
because $k_{r,\gamma_1}(\gamma_2) \ge 1$,
\[
  \left|\sum_{(\gamma_1, \gamma_2) \in \Gamma_{1,k_{n,\gamma_1}}}
  \E [X_{\gamma_1} X_{\gamma_2} \mid \sigma_U, \sigma_{U'}]\right|
  \le |\Gamma_{1,k_{n,\gamma_1}}| s^{2k} n^{-k - k_{n,\gamma_1} - 1 + o(1)}
  \le s^{2k} n^{-3 + o(1)}.
\]
Summing over the $k$ possible values of $k_{n,\gamma_1}$ adds another
$n^{o(1)}$ factor, and we conclude that $\Gamma_1$ is a lower order term.

Finally, we control $\Gamma_2$.
For each pair $(\gamma_1, \gamma_2) \in \Gamma_2$, we may create a
new path $\gamma$ by joining the end
of $\gamma_1$ to the beginning of $\gamma_2$. Then $\gamma$
has length $2k+1$ and $k_r(\gamma) \ge k^*$.
Note that $|X_{\gamma}| \ge \frac 1n |X_{\gamma_1} X_{\gamma_2}|$ because
the new edge that we added always has $|W_e| \ge \frac 1n$. Hence,
\[
  \left|\sum_{(\gamma_1, \gamma_2) \in \Gamma_{2}}
  \E [X_{\gamma_1} X_{\gamma_2} \mid \sigma_U, \sigma_{U'}]\right|
  \le n \sum_\gamma \E [|X_\gamma| \mid \sigma_U, \sigma_{U'}],
\]
where the second sum ranges over all $\gamma$ of length $2k+1$
satisfying $k_r(\gamma) \ge k^*$. But by Corollary~\ref{cor:X-many-returns},
the last quantity is at most $n^{-3 + o(1)}$, and so $\Gamma_2$
contributes a lower order term.

\section{Weighted sums over complicated paths}
\label{sec:tangled}

In this section, we will prove~\eqref{eq:thm-bad-paths} by controlling
the sum of $X_\gamma$ over all non-self-avoiding paths
(in other words, we will show that the random variable
$|N_{u,v} - Y_{u,v}|$ is small).
The basic idea is the following: letting $\Xi$ be the event that
the graph is $\ell$-tangle-free, we will show that
$1_\Xi (N_{u,v} - Y_{u,v})$ has a small second moment.
We will do this by constructing, for each path $\gamma$, a cover
of $\Xi$ containing events of the form ``the edges in
$F$ are banned,'' where $F$ is some subset of possible edges.
On each of these events, the fact that some edges are banned
will help us show that $X_\gamma$ is small.

For the rest of the subsection, we fix the following notation:
Let $\gamma$ be a path of length $k$ with $t$ $\ell$-tangles;
let $k_r = k_r(\gamma)$ and $k_n = k_n(\gamma)$.
Take $\zeta^{(1)}, \dots, \zeta^{(r)}$ to be the
canonical SAW-decomposition
of $\gamma$, and let $\Vends$ be its set of endpoints.
Let $\sZeta = \sZeta(\gamma)$ be the collection of $\zeta^{(i)}$ that
have length at most $4\ell$ and let $\lZeta = \lZeta(\gamma)$ be the
other $\zeta^{(i)}$. We write $E(\sZeta) = \bigcup_{\zeta \in \sZeta} \zeta$,
and similarly for $E(\lZeta)$. Define
\[
  \calF = \calF(\gamma) = \{F \subset E(\sZeta): \sum_{e \in F} m_e \ge t
  \text{ and } |F| \le k_r - 1\}.
\]
For $F \subset E(\sZeta)$, let $\Omega_F$ be the event that no
edge of $F$ appears in $G$, and let
\[
  \Omega_\calF = \bigcup_{F \in \calF} \Omega_F.
\]
For some intuition on $\Omega_{\calF(\gamma)}$, note that if
$\gamma$ has no $\ell$-tangles then $t(\gamma) = 0$ and so
$\emptyset \in \calF(\gamma)$. Hence, $\Omega_{\calF(\gamma)}$ is
the entire probability space.
On the other hand, if $\gamma$ is a figure-eight-shaped graph
consisting of two short (e.g.\ length $\ell$) cycles that share a vertex,
then
$\calF(\gamma)$ is the collection $\{\{e\}: e \in E(\gamma)\}$,
and $\Omega_{\calF(\gamma)}$ is the event that some edge in $E(\gamma)$
fails to appear.

Recall that $\Xi$ is the event that $G$ contains no $\ell$-tangles.
\begin{lemma}\label{lem:cover-xi}
  For any path $\gamma$,
  $\Xi \subset \Omega_{\calF(\gamma)}$.
\end{lemma}

\begin{proof}
  Fix the path $\gamma$.
  With the notation preceding the lemma, take some $G \in \Xi$
  and let $F$ be a minimal subset of $E(\sZeta(\gamma)) \setminus E(G)$ such that 
  $\gamma \setminus F$ has no $\ell$-tangles;
  to see that $F$ is well-defined, note that there exist
  subsets $F'$ of $E(\sZeta(\gamma)) \setminus E(G)$ such that
  $\gamma \setminus F'$ has no $\ell$-tangles: since $G$ is $\ell$-tangle-free,
  $E(\sZeta(\gamma)) \setminus E(G)$ is such a set.
  Since $F \cap E(G) = \emptyset$, clearly $G \in \Omega_F$.
  We claim that $F \in \calF(\gamma)$; this will imply that
  $G \in \Omega_{\calF(\gamma)}$, completing the proof.

  Since $\gamma \setminus F$ is $\ell$-tangle-free, the fact that
  $t$ is the number of $\ell$-tangles in $\gamma$ implies that
  $\sum_{e \in F} m_e \ge t$. In order to see that $|F| \le k_r - 1$,
  we claim that no connected component of $\gamma \setminus F$
  is a tree, and that therefore $\gamma\setminus F$ has at least
  as many edges as vertices.
  Indeed, $\gamma$ is connected and so if $\gamma \setminus F$ has
  some connected component $\gamma_1$ that is a tree then $F$ must contain
  an edge connecting $\gamma_1$ to some other component of
  $\gamma \setminus F$. Adding this edge back into $\gamma \setminus F$
  cannot introduce any tangles, contradicting the assumption that
  $F$ was minimal.
  We conclude that $\gamma \setminus F$ has at least as many edges
  as vertices, and so $|F| \le |E(\gamma)| - |V(\gamma)| = k_r - 1$.
\end{proof}

\begin{lemma}\label{lem:xi-small}
  If $\ell \log d = o(\log n)$ then
  $\P[\Xi] \ge 1 - n^{-1 + o(1)}$.
\end{lemma}
\begin{proof}
  Let $H$ be any fixed graph with $m$ vertices and $m+1$ edges.
  There are at most $n^m$ ways to embed $H$ into $G$, and for each
  of those embeddings, the probability that all edges in $H$ appear
  is at most $(2d/n)^m$. By a union bound, the probability that
  $H$ is a subgraph of $G$ is at most $n^{-1} (2d)^m$.

  Now, if $G$ contains an $\ell$-tangle then it has some neighborhood
  of radius $\ell$ containing two cycles. By taking these two cycles
  and (if necessary) the shortest path connecting them, we obtain
  a subgraph of $H$ of $G$ with $m \le 4\ell$ vertices, $m+1$ edges,
  and no vertices of degree $1$. Up to isomorphism, there are no
  more than $O(\ell^3)$ such graphs $H$: indeed, all vertices
  of $H$ have degree two, except for either two that have degree three
  or one that has degree four. In either case, the graph is determined
  up to isomorphism
  by specifying the lengths of the chains that connect these vertices
  of higher degree. These lengths are all at most $4\ell$, and there are
  at most three of them to choose. By taking a union bound over all
  such $H$ and applying the argument of the previous paragraph,
  $1 - \P[\Xi] \le O(\ell^3 n^{-1} (2d)^{4\ell}) = n^{-1 + o(1)}$.
\end{proof}

Before proceeding to bound the weight of non-self-avoiding paths,
we present one more preliminary lemma. Because we will take a second
moment, we will need to handle pairs of non-self-avoiding paths.
In order to do so, we need to interpret the condition
$\sum_{e \in F} m_e \ge t$ in the definition of
$\calF(\gamma)$ for pairs of paths.
In the following lemma we deal with multiple paths, so we
will write $m_e(\gamma)$ for the number of times that the path
$\gamma$ crosses the edge $e$.

\begin{lemma}\label{lem:cutting-tangles}
  Let $\gamma_1$ and $\gamma_2$ be two paths from $u$ to $v$ of length $k$.
  Let $\gamma$ be the path from $u$ to $u$ obtained by first following
  $\gamma_1$ and then following the reversal of $\gamma_2$.
  For any $F_1 \in \calF(\gamma_1)$ and $F_2 \in \calF(\gamma_2)$,
  \[
    \sum_{e \in F_1 \cup F_2} m_e(\gamma)
    \ge t(\gamma_1) + t(\gamma_2) + |F_1 \cup F_2| - k_r(\gamma) + 1.
  \]
\end{lemma}

\begin{proof}
  Let $F = F_1 \cup F_2$, $F_1' = (F_1 \setminus F_2) \cap E(\gamma_2)$,
  and $F_2' = (F_2 \setminus F_1) \cap E(\gamma_1)$.
  Then set $H = F \setminus (F_1' \cup F_2')$.
  Recall that $\gamma_i \setminus F_i$ has at least as many edges
  as vertices.
  \[
    \gamma \setminus H = (\gamma_1 \setminus F_1) \cup (\gamma_2 \setminus F_2)
  \]
  also has at least as many edges as vertices.
  Hence, $|H| \le k_r(\gamma) - 1$.
  
  For every $e \in F_i'$, we have $m_e(\gamma) \ge m_e(\gamma_i) + 1$.
  Hence,
  \begin{align*}
    \sum_{e \in F} m_e(\gamma)
    &= \sum_{e \in H} m_e(\gamma)
    + \sum_{e \in F_1'} m_e(\gamma)
    + \sum_{e \in F_2'} m_e(\gamma) \\
    &\ge  \sum_{e \in H} (m_e(\gamma_1) + m_e(\gamma_2))
    + \sum_{e \in F_1'} m_e(\gamma_1) + \sum_{e \in F_2'} m_e(\gamma_2)
    + |F_1'| + |F_2'| \\
    &\ge t(\gamma_1) + t(\gamma_2) + |F_1'| + |F_2'| \\
    &\ge t(\gamma_1) + t(\gamma_2) + |F| - k_r(\gamma) + 1.
    \qedhere
  \end{align*}
\end{proof}

\subsection{The weight of a complicated path}

Let $\gamma_1$ and $\gamma_2$ be non-backtracking paths from $u$ to
$v$ of length $k$ that are not self-avoiding. That is,
$k_r(\gamma_1), k_r(\gamma_2) \ge 1$. Let $\gamma$ be the path
obtained by first following $\gamma_1$ and then following $\gamma_2$
backwards.
Let $t(\gamma_i)$ be the number of tangles in $\gamma_i$.

Recall from~\eqref{eq:k*} that $k^*$ is a constant (depending on
$s$ and $d$) such that paths with $k_r > k^*$ are irrelevant.

\begin{lemma}\label{lem:one-tangled-path}
  Suppose that $k_r(\gamma) \le k^*$.
  Take $U \subset V$.
  If $\ell |U| = o(\log n)$ and $t = t(\gamma_1) + t(\gamma_2)$ then
  for any $F_1 \in \calF(\gamma_1)$ and $F_2 \in \calF(\gamma_2)$,
  \[
    \left|
      \E \left[
        1_{\Omega_{F_1}} 1_{\Omega_{F_2}} X_\gamma
        \mid \sigma_U
      \right]
    \right| \le
    s^{2k} n^{-k_n(\gamma) - t - 1 + o(1)},
  \]
  uniformly over $\gamma$ and $\sigma_U$.
\end{lemma}

\begin{proof}
  Let $\zeta^{(1)}, \dots, \zeta^{(r)}$ be the canonical
  SAW-decomposition of $\gamma$ with respect to $U$, and let $\Vends$ be its set of endpoints.
  Let $m_e$ be the number of times that the edge $e$ is traversed in $\gamma$.

  Let $K$ be the collection of $\zeta^{(i)}$ that contain
  some edge of either $E(\sZeta(\gamma_1))$ or $E(\sZeta(\gamma_2))$;
  note that $E(K) = E(\sZeta(\gamma_1)) \cup E(\sZeta(\gamma_2))$, but that $K$
  is not the same thing as $\sZeta(\gamma)$, because $\gamma$ may have some
  short segments that were not short in either $\gamma_1$ or $\gamma_2$.
  Let $L$ be the collection of $\zeta^{(i)}$ that do not
  belong to $K$.
  We split $X_\gamma = X_K X_L$, where $X_K = \prod_{e \in E(K)} W_e^{m_e}$
  and $X_L = \prod_{e \in E(L)} W_e^{m_e}$. Note that $\Omega_{\calF(\gamma_i)}$ for $i = 1, 2$ only depend on the edges in $E(K)$; hence,
  \begin{align}
    \left| \E \left[
        1_{\Omega_{\calF(\gamma_1)}} 1_{\Omega_{\calF(\gamma_2)}} X_\gamma
        \mid \sigma_U
    \right] \right|
    &= \left| \E \left[
        1_{\Omega_{\calF(\gamma_1)}} 1_{\Omega_{\calF(\gamma_2)}} X_K
        \mid \sigma_U
    \right]
    \E \left[ X_L
        \mid \sigma_U
    \right] \right| \notag \\
    &\le \E \left[
        1_{\Omega_{\calF(\gamma_1)}} 1_{\Omega_{\calF(\gamma_2)}} |X_K|
        \mid \sigma_U
    \right]
    \left| \E \left[ X_L
        \mid \sigma_U
    \right] \right|\label{eq:split-K-L}
  \end{align}
  The term involving $X_L$ may be bounded by
  Lemma~\ref{lem:non-tangled-segments} (recalling that the combined path
  $\gamma$ has length $2k$):
  \begin{equation}\label{eq:X_L-bound}
    \left|\E [X_L \mid \sigma_U]\right|
    \le (1 + n^{-1 + o(1)}) 2^r n^{-|E(L)|} s^{2k}.
  \end{equation}

  Next, we turn to the first term of~\eqref{eq:split-K-L}.
Recall that $\Omega_F$ is the event that no edge in $F$ appears. Hence,
if $F_1 \in \calF(\gamma_1)$, $F_2 \in \calF(\gamma_2)$, and $F = F_1 \cup F_2$ then
\begin{align*}
  \E [1_{\Omega_F} |X_K| \mid \sigma_U]
  &= \prod_{e \in F} (d/n)^{m_e}
  \prod_{e \in E(K) \setminus F} \E [|W_e|^{m_e} \mid \sigma_U] \\
  &\le (d/n)^{|F| + t(\gamma_1) + t(\gamma_2) - k_r(\gamma) + 1} (2d/n)^{|E(K)| - |F|} \\
  &\le (2d/n)^{|E(K)| + t - k_r(\gamma) + 1},
\end{align*}
where we applied Lemma~\ref{lem:cutting-tangles}
in the first inequality above. Hence,
\begin{align*}
  \E  \left[| 1_{\Omega_\calF(\gamma_1)} 1_{\Omega_{\calF(\gamma_2)}} X_K | \mid \sigma_U\right]
  &\le \E \left[ \sum_{F_1 \in \calF(\gamma_1)} \sum_{F_2 \in \calF(\gamma_2)} 1_{\Omega_{F_1 \cup F_2}} |X_K| \ \big\mid\  \sigma_U\right] \\
  &\le |\calF(\gamma_1)| |\calF(\gamma_2)|
  (2d/n)^{|E(K)| + t - k_r(\gamma) + 1} \\
    &\le k^{k_r(\gamma_1) + k_r(\gamma_2)}
  (2d/n)^{|E(K)| + t - k_r(\gamma) + 1}
\end{align*}
Combining this with~\eqref{eq:split-K-L} and~\eqref{eq:X_L-bound},
\[
    \left| \E \left[
        1_{\Omega_{\calF(\gamma_1)}} 1_{\Omega_{\calF(\gamma_2)}} X_\gamma
        \mid \sigma_U
    \right] \right|
    \le (1 + o(1))
        k^{k_r(\gamma_1) + k_r(\gamma_2)}
        (2d)^{|E(K)| + t + r} s^{2k} n^{-|E(K)| - |E(L)| - t + k_r(\gamma) - 1}.
\]
Now, $k_r(\gamma_1) + k_r(\gamma_2)$ is at most the number of returning
edges in $\gamma$, which is at most $O(1)$. Moreover, $|E(K)| \le 4 \ell r$,
$t \le r$,
and $r \le O(1) + |U| + 1$; hence, the quantity above is bounded by
\[
  s^{2k} n^{-|E(K)|-|E(L)|-t + k_r(\gamma) - 1 + o(1)}.
\]
Now,
$|E(K)| + |E(L)|$ is the number of distinct edges traversed by $\gamma$,
which is also equal to $k_n(\gamma) + k_r(\gamma)$.
Applying this in the exponent of $n$ completes the proof.
\end{proof}

\subsection{Proof of~\eqref{eq:thm-bad-paths}}

We will now combine Lemma~\ref{lem:one-tangled-path}
with our earlier bounds on the number of paths
(Lemma~\ref{lem:few-tangles}) to show that the total weight
of non-self-avoiding paths is negligible on the event
that the graph contains no tangles.

\begin{lemma}\label{lem:tangled-weights}
  Fix vertices $u$ and $v$, and
  let $\badpaths$ be the set of non-backtracking, non-self-avoiding
  paths from $u$ to $v$ of length $k$. For any set $U \subset V$
  with $|U| = n^{o(1)}$, if $k/\ell = o(\log n / \log \log n)$ then
  uniformly over all labellings $\sigma_U$ on $U$,
  \[
    \E \left[1_{\Xi} \Big(\sum_{\gamma \in \badpaths} X_\gamma \Big)^2\mid \sigma_U \right] \le s^{2k} n^{-3 + o(1)}.
  \]
\end{lemma}

Before proving Lemma~\ref{lem:tangled-weights},
note that because $\sum_{\gamma \in \badpaths} X_\gamma$
is the same as $N_{u,v}^{(k)} - Y_{u,v}$ in
Theorem~\ref{thm:path-weights-main},~\eqref{eq:thm-bad-paths}
follows from Lemma~\ref{lem:tangled-weights}, Chebyshev's inequality,
and the fact that (Lemma~\ref{lem:xi-small})
$\P(\Xi) = 1 - n^{-1 + o(1)}$.

\begin{proof}[Proof of Lemma~\ref{lem:tangled-weights}]
 We may bound
  \[
    1_{\Xi} \Big(\sum_{\gamma \in \badpaths} X_\gamma \Big)^2
    \le \Big(\sum_{\gamma \in \badpaths} 1_{\Omega_\calF(\gamma)} X_\gamma \Big)^2
  \]
  because both sides are non-negative and, by Lemma~\ref{lem:cover-xi},
  they agree whenever the left hand side is non-zero. Next,
  we expand the sum above as
  \[
    \sum_{\gamma_1, \gamma_2 \in \badpaths}
    1_{\Omega_\calF(\gamma_1)}
    1_{\Omega_\calF(\gamma_2)} X_{\gamma_1} X_{\gamma_2}
  \]
  Combining the last two displayed equations,
  \[
    \E\left[ 1_\Xi \big(\sum_{\gamma \in \badpaths} X_\gamma\big)^2 \ \Big\mid\ \sigma_U\right]
    \le \sum_{\gamma_1 \in \badpaths}
    \sum_{\gamma_2 \in \badpaths}
    \E[1_{\Omega_\calF(\gamma_1)}
       1_{\Omega_\calF(\gamma_2)} X_{\gamma_1} X_{\gamma_2} \mid \sigma_U].
  \]
  Let $\gamma = \gamma(\gamma_1, \gamma_2)$ be $\gamma_1$ concatenated
  with the reverse of $\gamma_2$. Let $\Lambda_1 \subset \badpaths \times\badpaths$
  be the set of pairs $(\gamma_1, \gamma_2)$ such that
  $\gamma(\gamma_1, \gamma_2)$ has more than $k^*$ returning edges.
  Let $\Lambda_2 \subset (\badpaths \times \badpaths) \setminus \Lambda_1$ be the set of pairs
  $(\gamma_1, \gamma_2)$ such that $|(V(\gamma_1) \cup V(\gamma_2)) \cap U| \ge 2 \sqrt {\log n}$. Let $\Lambda = \Lambda_1 \cup \Lambda_2$. By 
  Corollary~\ref{cor:X-many-returns},
  \[
    \sum_{(\gamma_1, \gamma_2) \in \Lambda_1}
    |\E[1_{\Omega_\calF(\gamma_1)}
       1_{\Omega_\calF(\gamma_2)} X_{\gamma_1} X_{\gamma_2} \mid \sigma_U]|
       \le n^{-4 + o(1)}.
  \]
  Since $|U| = n^{o(1)} \le n^{1/2}$ for large enough $n$, the fraction
  of $\gamma_1 \in \badpaths$ such that $|V(\gamma_1) \cap U| \ge \sqrt{\log n}$ is at most $n^{-c \sqrt{\log n}}$ for some constant $c > 0$.
  By Lemma~\ref{lem:constant-returns},
  \begin{align*}
    \sum_{(\gamma_1, \gamma_2) \in \Lambda_2}
    |\E[1_{\Omega_\calF(\gamma_1)}
       1_{\Omega_\calF(\gamma_2)} X_{\gamma_1} X_{\gamma_2} \mid \sigma_U]|
       &\le
    \sum_{(\gamma_1, \gamma_2) \in \Lambda_2}
\left(\frac{2d}{n}\right)^{k_n(\gamma(\gamma_1, \gamma_2)) + k_r(\gamma(\gamma_1, \gamma_2))} \\
&\le (2d)^k n^{-c \sqrt{\log n} + O(1)} \le n^{-4 + o(1)},
  \end{align*}
  and so it remains to prove that
  \begin{equation}\label{eq:tangled-weights-target-1}
    \sum_{\gamma_1 \in \badpaths}
    \sum_{\gamma_2 \in \badpaths}
    1_{\{(\gamma_1, \gamma_2) \not \in \Lambda\}}
    \E[1_{\Omega_\calF(\gamma_1)}
       1_{\Omega_\calF(\gamma_2)} X_{\gamma_1} X_{\gamma_2} \mid \sigma_U]
       \le s^{2k} n^{-3 + o(1)}
  \end{equation}
  We will further split this sum according to the number
  of tangles in $\gamma_1$ and $\gamma_2$; that is, we
  define $\badpaths_t$ to be the set of $\gamma_1 \in \badpaths$
  with $t(\gamma_1) = t$. We will show that for any $t_1$ and $t_2$,
  \begin{equation}\label{eq:tangled-weights-target-2}
    \sum_{\gamma_1 \in \badpaths_{t_1}}
    \sum_{\gamma_2 \in \badpaths_{t_2}}
    1_{\{(\gamma_1, \gamma_2) \not \in \Lambda_1\}}
    \E[1_{\Omega_\calF(\gamma_1)}
       1_{\Omega_\calF(\gamma_2)} X_{\gamma_1} X_{\gamma_2} \mid \sigma_U]
       \le s^{2k} n^{-3 + o(1)}.
  \end{equation}
  Summing over the $k = n^{o(1)}$ possible values of $t_1$ and $t_2$,
  this will imply~\eqref{eq:tangled-weights-target-1} and complete
  the proof.

  To control~\eqref{eq:tangled-weights-target-2},
  fix $\gamma_1$ and consider the sum over
  $\gamma_2$. By Lemma~\ref{lem:few-tangles}, there are at most
  $n^{k_n + t_2 - 1 + o(1)}$ choices of $\gamma_2 \in \badpaths_{t_2}$
  that satisfy
  $k_n(\gamma_2) = k_n$ (denote this set by $\badpaths_{t_2,k_n}$).
  Note that the fraction of $\gamma_2 \in \badpaths_{t_2,k_n}$
  satisfying $k_n(\gamma) = k_n(\gamma_1) + k_n(\gamma_2) - m$
  is at most $k^{2m} (n-2k)^{-m}$. Indeed,
  $m = k_n(\gamma_1) + k_n(\gamma_2) - k_n(\gamma)$ is the number of
  edges that were new in $\gamma_2$ but not in $\gamma$. There are
  at most $k^m$ ways to choose which edges in $\gamma_2$ will no
  longer be new and each one has at most $k^m$ choices for a non-new
  step, versus at least $(n-2k)^m$ choices for a new step.
  Set $\badpaths_{t_2,k_n,m,\gamma_1}$ to be the paths
  $\gamma_2 \in \badpaths_{t_2,k_n}$ satisfying $k_n(\gamma) = k_n(\gamma_1) + k_n(\gamma_2) - m$. Note that if $(\gamma_1, \gamma_2) \not \in \Lambda$ then
  the total number of returning edges in $\gamma$ is at most $k^* = O(1)$,
  and the number of vertices in $U$ intersecting $V(\gamma)$ is at
  most $2 \sqrt{\log n}$.
  Hence, Lemma~\ref{lem:one-tangled-path} applied
  with $U = U \cap V(\gamma)$
  implies that for any $\gamma_1 \in \badpaths_{t_1}$ and any $m$,
  \begin{align*}
    \sum_{\gamma_2 \in \badpaths_{t_2,k_n,m,\gamma_1}}
    1_{\{(\gamma_1, \gamma_2) \not \in \Lambda\}}
    \E[1_{\Omega_\calF(\gamma_1)}
       1_{\Omega_\calF(\gamma_2)} X_{\gamma_1} X_{\gamma_2} \mid \sigma_U]
    &\le |\badpaths_{t,k_n,m,\gamma_1}|
       s^{2k} n^{-k_n(\gamma_1) - k_n(\gamma_2) + m - t_1 - t_2 - 1 + o(1)} \\
    &\le |\badpaths_{t,k_n}|
       s^{2k} n^{-k_n(\gamma_1) - k_n(\gamma_2) - t_1 - t_2 - 1 + o(1)} \\
       &\le s^{2k} n^{-k_n(\gamma_1) - t_1 - 2 + o(1)}.
  \end{align*}
  Taking the sum over $k_n \le k$ and $m \le k$ only contributes
  a factor of $n^{o(1)}$; hence,
  \[
    \sum_{\gamma_2 \in \badpaths_{t_2}}
    1_{\{(\gamma_1, \gamma_2) \not \in \Lambda\}}
    \E[1_{\Omega_\calF(\gamma_1)}
       1_{\Omega_\calF(\gamma_2)} X_{\gamma_1} X_{\gamma_2} \mid \sigma_U]
     \le s^{2k} n^{-k_n(\gamma_1) - t_1 - 2 + o(1)}.
  \]
  Summing over the $n^{k_n + t_1 - 1 + o(1)}$ possible $\gamma_1 \in \badpaths_{t_1,k_n}$
  and then over the $k$ possible values of $k_n$, we see that
  the right hand side of~\eqref{eq:tangled-weights-target-2} 
  is bounded by $s^{2k} n^{-3 + o(1)}$, as claimed.
\end{proof}

Finally, note that we have finished the proof of
Theorem~\ref{thm:path-weights-main}. Indeed,
we proved~\eqref{eq:thm-first-moment} at the beginning of
Section~\ref{sec:saw-paths-proofs},~\eqref{eq:thm-second-moment}
in Section~\ref{sec:second-mom},~\eqref{eq:thm-cross-moments} in
Section~\ref{sec:cross-mom}, and we just proved~\eqref{eq:thm-bad-paths}.

\section*{Acknowledgments}
The authors are grateful to Cris Moore and Lenka Zdeborov\'a for stimulating and interesting discussions on many aspects of the block model.
They also thank the Charles Bordenave and the anonymous referees for
pointing out several simplifications and corrections.
\bibliography{block-model,all}
\bibliographystyle{plain}
\end{document}